\documentclass{article}

\usepackage[T1]{fontenc}
\usepackage{amssymb}
\usepackage{amsmath}
\usepackage{amsthm}
\usepackage{enumerate}
\usepackage{algorithm}
\usepackage{algorithmic}
\usepackage{subfig}
\usepackage{ifthen}
\usepackage{graphicx,color}
\newcommand{\1}{\ensuremath{\mathbf{1}}}

\theoremstyle{plain} \newtheorem{theorem}{Theorem}[section]
\theoremstyle{plain} \newtheorem{proposition}[theorem]{Proposition}
\theoremstyle{plain} 
\theoremstyle{plain} \newtheorem{lemma}[theorem]{Lemma}
\theoremstyle{remark}

\newcounter{hypA}
\newenvironment{hypA}{\refstepcounter{hypA}\begin{itemize}
\item[{\bf A\arabic{hypA}}]}{\end{itemize}}

\newcommand{\normL}[2]{\left\| #2 \right\|_{#1}}
\newcommand{\osc}{\mathrm{osc}}
\newcommand{\rme}{\mathrm{e}}
\newcommand{\eqdef}{\ensuremath{\stackrel{\mathrm{def}}{=}}}

\def\Xset{\mathbb{X}}
\def\Yset{\mathbb{Y}}
\def\Zset{\mathbb{Z}}

\newcommand{\pscal}[2]{\left\langle #1, #2 \right\rangle}

\newcommand{\Rset}{\mathbb{R}}

\newcommand{\eqsp}{\;}

\newcommand{\param}{\theta}
\newcommand{\paramset}{\Theta}
\newcommand{\rmd}{\mathrm{d}}
\newcommand{\rmB}{\mathrm{B}}

\newcommand{\rmL}{\mathrm{L}}
\newcommand{\bfy}{\mathbf{y}}
\newcommand{\bfY}{\mathbf{Y}}

\newcommand{\CExpparam}[3]{\mathbb{E}_{#3}\left[#1 \middle| #2\right]}

\newcommand{\CExp}[2]{\mathbb{E}_{}\left[#1\middle | #2\right]}

\newcommand{\PP}{\mathbb{P}}
\newcommand{\Staset}{\mathcal{L}}

\newcommand{\CE}[4][]
{
\ifthenelse{\equal{#1}{}}{\mathbb{E}^{#2}_{#3}\left[#4\right]}{\mathbb{E}^{#2}_{#3}\left[#4\middle | #1\right]}
}

\newcommand{\lpnorm}[2]{\ensuremath{\left\| #1 \right\|_{#2}}} 

\newcommand{\smoothfunc}[2]{\Phi_{#2}^{#1}}
\newcommand{\lyap}{\mathrm{W}}

\newcommand{\ps}[1]{#1-\mathrm{a.s.}}
\newcommand{\limEMmap}{\mathrm{G}}
\newcommand{\limEMmapt}{\mathrm{R}}
\newcommand{\mapS}{\bar{\mathrm{S}}}
\newcommand{\PPim}{\PP}
\newcommand{\Sset}{\mathcal{S}}

\newcommand{\averageparam}[1][]%
{\ifthenelse{\equal{#1}{}}{\ensuremath{\widetilde{\theta}}}{\ensuremath{\widetilde{\theta}}_{#1}}
}

\newcommand{\mix}[1][]%
{\ifthenelse{\equal{#1}{a}}{\alpha}{\beta}
}

\newcommand{\epart}[2]{\xi_{#1}^{#2}}
\newcommand{\ewght}[2]{\omega_{#1}^{#2}}

\newcommand{\GisCentered}{\cite[Lemma 3]{douc:garivier:moulines:olsson:2010}}
\newcommand{\GNorm}{\cite[Lemma 10]{douc:garivier:moulines:olsson:2010}}
\newcommand{\addfunc}[1]{\mathsf{S}_{#1}}

\def\sigmaX{\mathcal{X}}
\def\sigmaY{\mathcal{Y}}

\newcommand{\m}{\ensuremath{m}}
\newcommand{\adjfunc}[4][]
{\ifthenelse{\equal{#1}{}}{\ifthenelse{\equal{#4}{}}{\upsilon_{#2}}{\upsilon_{#2}(#4)}}
{\ifthenelse{\equal{#1}{smooth}}{\ifthenelse{\equal{#4}{}}{\tilde{\upsilon}_{#2}}{\tilde{\upsilon}_{#2}(#4)}}
{\ifthenelse{\equal{#1}{fully}}{\ifthenelse{\equal{#4}{}}{\upsilon^\star_{#2}}{\upsilon^\star_{#2}(#4)}}{\mathrm{erreur}}}}}
\newcommand{\kiss}[3][]
{\ifthenelse{\equal{#1}{}}{q_{#2}}
{\ifthenelse{\equal{#1}{fully}}{q^{\star}_{#2}}
{\ifthenelse{\equal{#1}{smooth}}{\tilde{r}_{#2}}{\mathrm{erreur}}}}}

\newcommand{\Kiss}[3][]
{\ifthenelse{\equal{#1}{}}{P_{#2}}
{\ifthenelse{\equal{#1}{fully}}{P^{\star}_{#2}}
{\ifthenelse{\equal{#1}{smooth}}{\tilde{R}_{#2}}{\mathrm{erreur}}}}}

\newcommand{\XinitIS}[2][]{\ifthenelse{\equal{#1}{}}{\ensuremath{\rho_{#2}}}{\ensuremath{\check{\rho}_{#2}}}}
\newcommand{\filt}[2][]%
{
\ifthenelse{\equal{#1}{}}{\ensuremath{\phi_{#2}}}%
{\ifthenelse{\equal{#1}{hat}}{\ensuremath{\phi^{N}_{#2}}}
{\ifthenelse{\equal{#1}{tilde}}{\ensuremath{\tilde{\phi}^{N}_{#2}}}
{\ifthenelse{\equal{#1}{tar}}{\ensuremath{\phi^{N,\mathrm{t}}_{#2}}}
{\ifthenelse{\equal{#1}{aux}}{\ensuremath{\phi^{N,\mathrm{a}}_{#2}}}
}
}
}
}
}

\newcommand{\instrpostaux}[2]{\ensuremath{\pi_{#1#2}}}
\newcommand{\BK}[1]{\mathrm{B}_{#1}}
\newcommand{\post}[3][]%
{
\ifthenelse{\equal{#1}{}}{\ensuremath{\phi_{#2|#3}}}%
{\ifthenelse{\equal{#1}{hat}}{\ensuremath{\phi^{N}_{#2|#3}}}
{\ifthenelse{\equal{#1}{tilde}}{\ensuremath{\tilde{\phi}^{N}_{#2|#3}}}
{\ifthenelse{\equal{#1}{tar}}{\ensuremath{\phi^{N,\mathrm{t}}_{#2|#3}}}}
}
}
}
\newcommand{\chunk}[4][]%
{\ifthenelse{\equal{#1}{}}{\ensuremath{{#2}_{#3:#4}}}{\ensuremath{#2^#1}_{#3:#4}}
}
\newcommand{\sumwght}[2][]{%
\ifthenelse{\equal{#1}{}}{\ensuremath{\Omega_{#2}}}{\ensuremath{\Omega_{#2}^{#1}}}}

\begin{document}
  
\title{Convergence of a Particle-based Approximation of the Block Online Expectation Maximization Algorithm}
        \author{Sylvain Le Corff\footnote{This
    work is partially supported by the French National Research Agency, under
    the programs ANR-07-ROBO-0002 and ANR-08-BLAN-0218.} \footnote{LTCI, CNRS and TELECOM ParisTech, 46 rue Barrault 75634 Paris Cedex 13, France}\; and Gersende Fort\footnotemark[\value{footnote}]}

\maketitle

\begin{abstract}
  Online variants of the Expectation Maximization (EM) algorithm have recently been
  proposed to perform parameter inference with large data sets or data streams,
  in independent latent models and in hidden Markov models. Nevertheless, the
  convergence properties of these algorithms remain an open problem at least in
  the hidden Markov case.  This contribution deals with a new online EM
  algorithm which updates the parameter at some deterministic times. Some convergence results have been derived even in general latent models such as
  hidden Markov models. These properties rely on the assumption that some
  intermediate quantities are available in closed form or can be approximated by Monte Carlo methods when the Monte Carlo error vanishes rapidly enough. In this
  paper, we propose an algorithm which approximates these quantities using
  Sequential Monte Carlo methods. The convergence of this algorithm and of an
  averaged version is established and their performance is illustrated through
  Monte Carlo experiments.
\end{abstract}

\maketitle

\bigskip 

\hrule
\medskip 

This extended version of the paper ``Convergence of a Particle-based
Approximation of the Block Online Expectation Maximization Algorithm``, by S.
Le Corff and G. Fort, provides detailed proofs which have been omitted in the
submitted paper since they are very close to existing results. These additional
proofs are postponed to Appendix~\ref{app:extended:version}.

\medskip
\hrule

\bigskip 

\section{Introduction}
\label{sec:intro}
The Expectation Maximization (EM) algorithm is a well-known iterative algorithm
to solve maximum likelihood estimation in incomplete data
models, see ~\cite{dempster:laird:rubin:1977}.  Each iteration is decomposed into two
steps: in the E-step the conditional expectation of the complete log-likelihood
(log of the joint distribution of the hidden states and the observations) given
the observations is computed; and the M-step updates the parameter estimate. The EM algorithm is mostly practicable if the model belongs to the curved exponential family, see  \cite[Section $1.5$]{mclachlan:krishnan:1997} and \cite[Section $10.1$]{cappe:moulines:ryden:2005}, so that we assume below that our model belongs to this family. Under mild regularity conditions, this algorithm is
known to converge to the stationary points of the log-likelihood of the
observations, see \cite{wu:1983}.  However, the original EM algorithm
cannot be used to perform online estimation or when the inference task relies
on large data sets. Each iteration requires the whole data set and each
piece of data needs to be stored and scanned to produce a new parameter
estimate.  {\em Online} variants of the EM algorithm were first proposed
for independent and identically distributed (i.i.d.)  observations: \cite{cappe:moulines:2009} proposed to replace the original E-step by a
stochastic approximation using the new observation. Solutions have also been
proposed  in hidden Markov models
(HMM): ~\cite{cappe:2011} provides an algorithm for finite state-space HMM which
relies on recursive computations of the filtering distributions combined with a
stochastic approximation step.  Note that, since the state-space is finite,
deterministic approximations of these distributions are available.  This
algorithm has been extended to the case of general state-space models, the
approximations of the filtering distributions being handled with Sequential
Monte Carlo (SMC) algorithms, see~\cite{cappe:2009}, \cite{delmoral:doucet:singh:2010a} and \cite{lecorff:fort:moulines:2011}.
Unfortunately, it is quite challenging to address the asymptotic behavior of
these algorithms (in the HMM case) since the recursive computation of the
filtering distributions relies on approximations which are really difficult to
control.

In \cite{lecorff:fort:2011}, another online variant of the EM algorithm in HMM
is proposed, called the Block Online EM (BOEM) algorithm. In this case, the
data stream is decomposed into blocks of increasing sizes. Within each block,
the parameter estimate is kept fixed and the update occurs at the end of the
block.  This update is based on a single scan of the observations, so that it
is not required to store any block of observations.~\cite{lecorff:fort:2011}
provides results on the convergence and on the convergence rates of the
BOEM algorithms. These analyses are established when the E-step (computed on each block) is available in closed form and when it can be approximated using Monte Carlo methods, under an assumption on the $\rmL_{p}$-error of the Monte Carlo approximation. 

In this paper, we consider the case when the E-step of the  BOEM algorithm is computed with
SMC approximations: the filtering distributions are approximated using a set of
random weighted particles, see
\cite{cappe:moulines:ryden:2005} and \cite{delmoral:2004}. The Monte Carlo approximation
is based on an online variant of the Forward Filtering Backward Smoothing
algorithm (FFBS) proposed in \cite{cappe:2011} and \cite{delmoral:doucet:singh:2010a}.
This method is appealing for two reasons: first, it can be implemented forwards in time i.e. within a block, each observation is scanned once and
never stored and the approximation computed on each block does not require a
backward step - this is crucial in our online estimation framework. Secondly,
recent work on SMC approximations provides $\rmL_{p}$-mean control of the Monte
Carlo error, see e.g.  \cite{dubarry:lecorff:2011} and \cite{delmoral:doucet:singh:2010b}.
This control, combined with the results in \cite{lecorff:fort:2011}, sparks off
the convergence results and the convergence rates provided in this
contribution.

The paper is organized as follows: our new algorithm called the {\em Particle Block
  Online EM} algorithm (P-BOEM) is derived in
Section~\ref{sec:BOEM:description} together with an {\em averaged} version.
Section~\ref{sec:MCexperiments} is devoted to practical applications: the P-BOEM algorithm is
used to perform parameter inference in stochastic volatility models and in the
more challenging framework of the Simultaneous Localization And Mapping problem
(SLAM).  The convergence properties and the convergence rates of the P-BOEM algorithms are given in Section~\ref{sec:convergence}.

\section{The Particle Block Online EM algorithms}
\label{sec:BOEM:description}
In Section~\ref{subsec:BOEM:description:notation}, we fix notation that will
be used throughout this paper. We then derive our online algorithms in Sections
\ref{subsec:BOEM:description:PBOEM} and
\ref{subsec:BOEM:description:PBOEM:aver}. We finally detail, in
Section~\ref{subsec:BOEM:description:SMC}, the SMC procedure that makes our
algorithm a true online algorithm.

\subsection{Notations and Model assumptions}
\label{subsec:BOEM:description:notation}
A hidden Markov model on $\Xset\times \Yset$ is defined by an initial
distribution $\chi$ on $(\Xset,\sigmaX)$ and two families of transition
kernels. In this paper, the transition kernels are parametrized by
$\param\in\paramset$, where $\paramset \subseteq \Rset^{d_\param}$ is a compact set.  In the sequel,
the initial distribution $\chi$ on $(\Xset,\sigmaX)$ is assumed to be known and
fixed. The parameter is estimated online in the maximum likelihood sense using
a sequence of observations $\bfY$. Online maximum likelihood parameter
inference algorithms were proposed either with a gradient approach or an
EM approach. In the case of finite state-spaces HMM, \cite{legland:mevel:1997}
proposed a recursive maximum likelihood procedure.  The asymptotic properties
of this algorithm have recently been addressed in \cite{tadic:2010}. This
algorithm has been adapted to general state-spaces HMM with SMC methods (see
\cite{doucet:poyiadjis:singh:2009}). The main drawback of gradient methods is
the necessity to scale the gradient components. As an alternative to performing
online inference in HMM, online EM based algorithms have been proposed for
finite state-spaces (see \cite{cappe:2011}) or general state-spaces HMM (see
\cite{cappe:2009}, \cite{delmoral:doucet:singh:2010a} and \cite{lecorff:fort:2011}).
\cite{delmoral:doucet:singh:2010a} proposed a SMC method giving encouraging
experimental results. Nevertheless, it relies on a combination of stochastic
approximations and SMC computations so that its analysis is quite challenging.
In \cite{lecorff:fort:2011}, the convergence of an online EM based algorithm is
established. This algorithm requires either the exact computation of intermediate
quantities (available explicitly only in finite state-spaces HMM or in linear
Gaussian models) or the use of Monte Carlo methods to approximate these quantities. We propose to apply this algorithm to general models where these quantities are replaced by SMC approximations. We prove that the Monte Carlo error is controlled in such a way that the convergence properties of \cite{lecorff:fort:2011} hold for the P-BOEM algorithms.
 
We now detail the model assumptions. Consider a family of transition kernels
$\{m_\param(x,x') \rmd \lambda(x')\}_{\param\in\paramset}$ on
$\Xset\times\sigmaX$, where $\Xset$ is a general state-space equipped with a
countably generated $\sigma$-field $\sigmaX$, and $\lambda$ is a finite measure on $(\Xset,\sigmaX)$. Let $\{g_\param(x,y) \rmd
\nu(y)\}_{\param\in\paramset}$ be a family of transition kernels on $\Xset
\times \sigmaY$, where $\Yset$ is a general space endowed with a countably
generated $\sigma$-field $\sigmaY$ and $\nu$ is a measure on
$(\Yset,\sigmaY)$. Let $\bfY=\{\bfY_t\}_{t\in\Zset}$ be the observation process
defined on $\left(\Omega,\PPim,\mathcal{F}\right)$ and taking values in
$\Yset^\Zset$. The batch EM algorithm is an offline maximum likelihood procedure which
iteratively produces parameter estimates using the complete data log-likelihood
(log of the joint distribution of the observations and the states) and a fixed
set of observations, see \cite{dempster:laird:rubin:1977}. In the HMM context
presented above, given $T$ observations $\bfY_{1:T}$, the missing
data $x_{0:T}$ and a parameter $\param$, the complete data log-likelihood may
be written as (up to the initial distribution $\chi$ which is assumed to be
known)
\begin{equation}
  \label{eq:complete:loglikeli}
  \ell_{\param}(x_{0:T},\bfY_{1:T})\eqdef \sum_{t=1}^{T}\left\{\log m_{\param}(x_{t-1},x_{t})+\log g_{\param}(x_{t},\bfY_{t})\right\}\eqsp,
\end{equation}
where we use $x_{r:t}$ as a shorthand notation for the sequence $(x_r,\dots, x_t)$, $r \leq t$.
Each iteration of the batch EM algorithm is decomposed into two steps. The E-step
computes, for all $\param\in\paramset,$ an expectation of the complete data
log-likelihood under the conditional probability of the hidden states given the
observations and the current parameter estimate $\hat \param$.  In the HMM
context, due to the additive form of the complete data log-likelihood
(\ref{eq:complete:loglikeli}), the E-step is decomposed into $T$ expectations
under the conditional probabilities $\smoothfunc{\chi,0}{\hat
  \param,t,T}(\cdot,\bfy)$ where
\begin{equation}
\label{eq:define-Phi}
\smoothfunc{\chi,r}{\param,s,t}(h,\bfy) \eqdef \frac{\int \chi(\rmd x_r) \{
  \prod_{i=r}^{t-1} m_\param(x_i, x_{i+1})g_{\param}(x_{i+1},\bfy_{i+1}) \} \,
  h(x_{s-1},x_{s},\bfy_{s}) \, \rmd \lambda(x_{r+1:t})}{\int \chi(\rmd x_r)
  \{\prod_{i=r}^{t-1} m_\param(x_i, x_{i+1})g_{\param}(x_{i+1},\bfy_{i+1}) \}\,
  \rmd \lambda(x_{r+1:t})}\eqsp,
\end{equation}
for any bounded function $h$, any $\param \in\paramset$, any $r <s \leq t$ and
any sequence $\bfy \in \Yset^\Zset$. Then, given the current value of the
parameter $\hat\param$, the E-step amounts to computing the quantity
\begin{equation}
\label{eq:QEM}
Q_T(\param,\hat\param) \eqdef \frac{1}{T} \sum_{t=1}^{T}\smoothfunc{\chi,0}{\hat\param,t,T}\left(
  \log m_{\param}+\log g_{\param}, \bfY\right)\eqsp,
\end{equation}
for any $\param \in \paramset$. The M-step sets the new parameter estimate as a
maximum of this expectation over $\param$. 

The computation of $\param \mapsto Q_{T}(\param,\hat\param)$ for any $\param
\in \paramset$ is usually intractable except in the case of complete data likelihood belonging to the curved exponential family, see  \cite[Section $1.5$]{mclachlan:krishnan:1997} and \cite[Section $10.1$]{cappe:moulines:ryden:2005}. Therefore, in the sequel, the
following assumption is assumed to hold:
\begin{hypA} 
\label{assum:exp}
  \begin{enumerate}[(a)]
  \item \label{assum:exp:decomp}There exist continuous functions $\phi :
    \paramset \to \Rset$, $\psi : \paramset \to \Rset^d$ and $S: \Xset \times
    \Xset \times \Yset \to \Rset^d$ s.t.
\begin{equation*}
  \label{eq:exponential:family}
  \log m_\param(x,x') + \log g_\param(x',y) = \phi(\param) +
  \pscal{S(x,x,',y)}{\psi(\param)} \eqsp,
\end{equation*}
where $\pscal{\cdot}{\cdot}$ denotes the scalar product on $\Rset^d$. 
\item \label{assum:exp:convex} There exists an open subset $\Sset$ of $\Rset^d$
  that contains the convex hull of $S(\Xset \times \Xset \times \Yset)$.
\item \label{assum:exp:max}There exists a continuous function $\bar \param :
  \Sset \to \paramset$ s.t. for any $s \in \Sset$,
\[
\bar \param(s) = \mathrm{argmax}_{\param \in \paramset} \; \left\{ \phi(\param) +
  \pscal{s}{\psi(\param)} \right\} \eqsp.
\]
  \end{enumerate}
\end{hypA}
Under A\ref{assum:exp}, the quantity $Q_{T}(\param,\hat\param)$ defined by
\eqref{eq:QEM} becomes
\begin{equation}
\label{eq:QEM-exp}
Q_{T}(\param,\hat\param) = \phi(\param) + \pscal{\frac{1}{T}\sum_{t=1}^{T}\smoothfunc{\chi,0}{\hat\param,t,T}\left(S, \bfY\right)}{\psi(\param)}\eqsp,
\end{equation}
so that the definition of the function $\param \mapsto Q_{T}(\param,\hat\param)$
requires the computation of an expectation
$\frac{1}{T}\sum_{t=1}^{T}\smoothfunc{\chi,0}{\hat\param,t,T}\left(S,
  \bfY\right)$ independently of $\param$. 

The M-step of the batch EM iteration amounts to computing
\[
\bar\param \left(\frac{1}{T}\sum_{t=1}^{T}\smoothfunc{\chi,0}{\hat\param,t,T}\left(S, \bfY\right)\right)\eqsp.
\]

This batch EM algorithm is designed for a fixed set of observations. A natural
extension of this algorithm to the online context is to define a sequence of
parameter estimates by
\[
\param_{t+1} = \mathrm{argmax}_{\param} \ Q_{t+1}(\param, \param_t) \eqsp.
\]
Unfortunately, the computation of $Q_{t+1}(\param, \param_t)$ requires the
whole set of observations to be stored and scanned for each estimation. For
large data sets the computation cost of the E-step makes it intractable in this
case.  To overcome this difficulty, several online variants of the batch EM
algorithm have been proposed, based on a recursive approximation of the
function $\param \mapsto Q_{t+1}(\cdot,\param_t)$ (see
\cite{cappe:2009}, \cite{delmoral:doucet:singh:2010a} and \cite{lecorff:fort:2011}). In this paper, we focus on the Block Online EM
(BOEM) algorithm, see \cite{lecorff:fort:2011}.

\subsection{Particle Block Online EM (P-BOEM)}
\label{subsec:BOEM:description:PBOEM}
The BOEM algorithm, introduced in \cite{lecorff:fort:2011}, is an online variant of the
EM algorithm. The observations are processed sequentially per block and the
parameter estimate is updated at the end of each block.  Let $\{\tau_k\}_{k \geq 1}$ be a sequence of positive integers denoting the length
of the blocks and set 
\begin{equation}
\label{eq:timeupdate}
T_n\eqdef \sum_{k=1}^n \tau_k\quad\mbox{and}\quad T_0 \eqdef 0\eqsp;
\end{equation}
$\{T_k\}_{k \geq 1}$ are the deterministic times at which the parameter updates
occur.  Define, for all integers $\tau>0$ and $T\geq 0$ and all
$\param\in\paramset$,
\begin{equation}
  \label{eq:rewrite:barS}
  \bar S_{\tau}^{\chi,T}(\param, \bfY) \eqdef
\frac{1}{\tau} \sum_{t=T+1}^{T+\tau}\smoothfunc{\chi,T}{\param,t,T+\tau}\left(
  S, \bfY\right)\eqsp.
\end{equation}
The quantity $\bar S_{\tau}^{\chi,T}(\param, \bfY)$ corresponds to the
intermediate quantity in \eqref{eq:QEM-exp} with the observations
$\bfY_{T+1:T+\tau}$.

The BOEM algorithm iteratively defines a sequence of parameter estimates $\{\param_n\}_{n \geq
  0}$ as follows: given the current parameter estimate $\param_{n}$,
\begin{enumerate}[(i)]
\item \label{BOEM:stat} compute the quantity $\bar S_{\tau_{n+1}}^{\chi,T_{n}}(\param_{n}, \bfY)$,
\item \label{BOEM:param}compute a candidate for the new value of the parameter:
  $\param_{n+1} = \bar\param\left(\bar
    S_{\tau_{n+1}}^{\chi,T_{n}}(\param_{n}, \bfY) \right)$,
\end{enumerate}
To make the exposition easier, we assume that the initial distribution $\chi$
is the same on each block.  The dependence of $\bar S_{\tau}^{\chi,T}(\param, \bfY)$
on $\chi$ is thus dropped from the notation for better clarity.

The quantity $\bar S_{\tau_{n+1}}^{T_{n}}(\param_{n},
\bfY)$ is available in closed form only in the case of linear
Gaussian models and HMM with finite state-spaces. In HMM with general
state-spaces $\bar S_{\tau_{n+1}}^{T_{n}}(\param_{n}, \bfY)$ cannot be computed
explicitly and we propose to compute an approximation of $\bar
S_{\tau_{n+1}}^{T_{n}}(\param_{n}, \bfY)$ using SMC algorithms thus yielding
the {\em Particle-BOEM} (P-BOEM) algorithm. Different methods can be
used to compute these approximations (see e.g.
\cite{delmoral:doucet:singh:2010b}, \cite{delmoral:doucet:singh:2010a} and  \cite{douc:garivier:moulines:olsson:2010}). We will discuss in
Section~\ref{subsec:BOEM:description:SMC} below some SMC approximations that
use the data sequentially.

\begin{list}{}{}
\item Denote by $\widetilde S_{n}(\param, \bfY)$ the SMC
approximation of $\bar S_{\tau_{n+1}}^{T_{n}}(\param, \bfY)$ computed with $N_{n+1}$ particles. The P-BOEM 
algorithm iteratively defines a sequence of parameter estimates $\{\param_n\}_{n \geq
  0}$ as follows: given the current parameter estimate $\param_{n}$,
\begin{enumerate}[(i)]
\item \label{PBOEM:stat} compute the quantity $\widetilde S_{n}(\param_n, \bfY)$,
\item \label{PBOEM:param}compute a candidate for the new value of the parameter:
  \[
  \param_{n+1} = \bar\param\left(\widetilde S_{n}(\param_{n}, \bfY) \right)\eqsp.
    \]
\end{enumerate}
\end{list}
We give in Algorithm~\ref{alg:PBOEM} lines $1$ to $9$ an algorithmic
description of the P-BOEM algorithm. Note that the idea of processing the observations by blocks is proposed in \cite{mclachlan:ng:2003} to fit a normal mixture model. The incremental EM algorithm discussed in \cite{mclachlan:ng:2003} is an alternative to the batch EM algorithm for very large data sets. Contrary to our framework, in the algorithm proposed by \cite{mclachlan:ng:2003}, the number of observations is fixed and the same observations are scanned several times.

\subsection{Averaged Particle Block Online EM}
\label{subsec:BOEM:description:PBOEM:aver}
Following the same lines as in \cite{lecorff:fort:2011}, we propose to replace
the P-BOEM sequence $\{\param_n \}_{n \geq 0}$ by an {\em averaged} sequence.
This new sequence can be computed recursively, simultaneously with the P-BOEM
sequence, and does not require additional storage of the data. The proposed
averaged P-BOEM algorithm is defined as follows (see also lines $5$ and
$6$ of Algorithm~\ref{alg:PBOEM}): 
the step \eqref{PBOEM:param} of the P-BOEM algorithm presented above is followed by
\begin{enumerate}[(i)]
\setcounter{enumi}{3}
\item \label{ABOEM:stat} compute the quantity 
\begin{equation}
\label{eq:PBonem:averaged}
\widetilde\Sigma_{n+1} = \frac{T_{n}}{T_{n+1}}\widetilde\Sigma_{n} + \frac{\tau_{n+1}}{T_{n+1}}\widetilde S_{n}(\param_{n}, \bfY)\eqsp,
\end{equation}
\item \label{ABOEM:param} define
\begin{equation}
  \label{eq:abonem:recursion}
\widetilde \param_{n+1} \eqdef \bar\param \left(\widetilde\Sigma_{n+1}\right)\eqsp.
\end{equation}
\end{enumerate}
We set $\widetilde\Sigma_{0} =0$ so that 
\begin{equation}
  \label{eq:tildeSigma:bis}
  \widetilde\Sigma_{n} = \frac{1}{T_n}
\sum_{j=1}^n \tau_j \, \widetilde S_{j-1}(\param_{j-1}, \bfY) \eqsp;
\end{equation}
we will prove in Section~\ref{sec:averaging} that the rate of convergence of
the averaged sequence $\{ \widetilde \param_n\}_{n \geq 0}$, computed from the
averaged statistics $\{\widetilde\Sigma_{n}\}_{n \geq 0}$, is better than the
non-averaged one. We will also observe this property in
Section~\ref{sec:MCexperiments} by comparing the variability of the P-BOEM and
the averaged P-BOEM sequences in numerical applications. 

\begin{algorithm}[h]
\caption{P-BOEM and averaged P-BOEM}
\label{alg:PBOEM}
\begin{algorithmic}
\REQUIRE $\theta_{0}$, $\{\tau_{n}\}_{n\geq 1}$, $\{N_{n}\}_{n\geq 1}$, $\{\bfY_{t}\}_{t\geq 0}\eqsp.$
\ENSURE $\{\param_{n}\}_{n\geq 0}$ and $\{\widetilde\param_{n}\}_{n\geq 0}\eqsp.$
\STATE Set $\widetilde\Sigma_{0} = 0$.
\FOR{all $i\geq 0$}
\STATE Compute sequentially $\widetilde S_{i}(\param_{i}, \bfY)\eqsp.$
\STATE Set $\param_{i+1} = \bar\param \left(\widetilde S_{i}(\param_{i}, \bfY)\right)\eqsp.$
\STATE Set 
\[
\widetilde\Sigma_{i+1} = \frac{T_{i}}{T_{i+1}}\widetilde\Sigma_{i} + \frac{\tau_{i+1}}{T_{i+1}}\widetilde S_{i}(\param_{i}, \bfY)\eqsp.
\]
\STATE Set $\widetilde\param_{i+1} = \bar\param \left(\widetilde \Sigma_{i+1}\right)\eqsp.$
\ENDFOR
\end{algorithmic}
\end{algorithm}

\subsection{The SMC approximation step}
\label{subsec:BOEM:description:SMC}
As the P-BOEM algorithm is an online algorithm, the SMC algorithm should use the data
sequentially: no backward pass is allowed to browse all the data at the end of
the block. Hence, the approximation is computed recursively within each block,
each observation being used once and never stored.  These SMC algorithms will
be referred to as {\em forward only SMC}.  We detail below a forward only SMC
algorithm for the computation of $\widetilde S_{n}(\param_{n}, \bfY)$ which has been proposed by
\cite{cappe:2011} (see also~\cite{delmoral:doucet:singh:2010a}).

For notational convenience, the dependence on $n$ is omitted. For block $n$, the algorithm below has to be applied with $(\tau,N)
\leftarrow (\tau_{n+1}, N_{n+1})$, $Y_{1:\tau} \leftarrow
Y_{T_{n}+1,T_{n}+\tau_{n+1}}$ and $\theta \leftarrow \theta_{n}$. 

The key property is to observe that
\begin{equation}
  \label{eq:SMCapprox:tool1}
  \bar S_{\tau}^{0}(\param, \bfY) = \phi_\tau^\param(R_{\param,\tau})
\end{equation}
where $\phi_{t}^\param$ is the filtering distribution at time $t$, and
the functions $R_{t,\param}:\Xset \to \Sset$, $1 \leq t \leq \tau$, satisfy the equations 
\begin{equation}
  \label{eq:SMCapprox:tool2}
R_{t,\param}(x) = \frac{1}{t} \mathrm{B}_{t}^\param\left(x, S(\cdot,x,Y_t)\right) + \frac{t-1}{t}
\mathrm{B}_{t}^\param\left(x, R_{t-1,\param}\right) \eqsp,
\end{equation}
where $\mathrm{B}_{t}^\param$ denotes the backward smoothing kernel at time $t$
\begin{equation}
  \label{eq:SMCapprox:tool3}
\mathrm{B}_{t}^\param(x, \rmd x') = \frac{m_\param(x',x) }{\int m_\param(u,x)
  \phi_{t-1}^\param(\rmd u) }\phi_{t-1}^\param(\rmd x') \eqsp.
\end{equation}
By convention, $R_{0,\param}(x) = 0$ and $\phi_{0}^\param= \chi$. A proof of
the equalities (\ref{eq:SMCapprox:tool1}) to (\ref{eq:SMCapprox:tool3}) can be
found in \cite{cappe:2011} and \cite{delmoral:doucet:singh:2010a}. Therefore, a careful
reading of Eqs~(\ref{eq:SMCapprox:tool1}) to (\ref{eq:SMCapprox:tool3}) shows
that, for an iterative particle approximation of $\bar S_{\tau}^{0}(\param,
\bfY)$, it is sufficient to update from time $t-1$ to $t$
\begin{enumerate}[(i)]
\item $N$ weighted samples
  $\left\{\left(\epart{t}{\ell},\ewght{t}{\ell}\right);
    \ell\in\{1,\dots,N\}\right\}$ used to approximate the filtering distribution
  $\phi_{t}^\param$.
\item the intermediate quantities $\{R_{t,\param}^{\ell}\}_{\ell=1}^{N}$,
  approximating the function $R_{t,\param}$ at point $x = \epart{t}{\ell}$, $\ell \in
  \{1,\cdots, N\}$.
\end{enumerate}
We describe below such an algorithm. An algorithmic description is also
provided in Appendix~\ref{app:alg}, Algorithm~\ref{alg:FSMC}. 

Given {\em instrumental} Markov transition kernels $\{\kiss{t}{}(x,x'), t \leq
\tau \}$ on $\Xset\times\mathcal{X}$ and adjustment multipliers
$\{\adjfunc{t}{}{}, t\leq \tau \}$, the procedure goes as follows:
  \begin{enumerate}[(i)]
  \item {\em line $1$ in Algorithm~\ref{alg:FSMC}:} sample independently $N$
    particles $\{\epart{0}{\ell}\}_{\ell=1}^{N}$ with the same distribution
    $\chi$.
  \item {\em line $6$ in Algorithm~\ref{alg:FSMC}:} at each time step
    $t\in\{1,\dots,\tau\}$, pairs $\{ (J_t^{\ell}, \epart{t}{\ell}) \}_{\ell =
      1}^{N}$ of indices and particles are sampled independently (conditionally
    to $Y_{1:t}$, $\param$ and $\{ (J_{t-1}^{\ell}, \epart{t-1}{\ell})
    \}_{\ell = 1}^{N}$) from the instrumental distribution:
\begin{equation} \label{eq:instrumental-distribution-filtering}
    \instrpostaux{t}{}(i, \rmd x) \propto \ewght{t-1}{i} \adjfunc{t}{}{\epart{t-1}{i}} \kiss{t}{}(\epart{t-1}{i},x)\lambda(\rmd x) \eqsp,
\end{equation}
on the product space $\{1, \dots, N\} \times \Xset$. 
For any $t\in\{1,\dots,\tau\}$ and any $\ell\in\{1,\dots,N\}$, $J_t^{\ell}$
denotes the index of the selected particle at time $t-1$ used to produce
$\epart{t}{\ell}$.
\item {\em line $7$ in
Algorithm~\ref{alg:FSMC}:} once the new particles $\{\epart{t}{\ell}\}_{\ell = 1}^{N}$ have been sampled,
their importance weights  $\{ \ewght{t}{\ell}\}_{\ell = 1}^{N}$ are computed.
\item {\em lines $8$ in Algorithm~\ref{alg:FSMC}:} update the intermediate
  quantities $\{R_{t,\param}^{\ell}\}_{\ell=1}^{N}$.
\end{enumerate}
If, for all $x\in\Xset$, $\adjfunc{t}{}{x}=1$ and if the kernels $\kiss{t}{}$
are chosen such that $\kiss{t}{} = m_{\param}$, lines $6$-$7$ in
Algorithm~\ref{alg:FSMC} are known as the {\em Bootstrap filter}. Other choices
of $\kiss{t}{}{}$ and $\adjfunc{t}{}{}$ can be made, see e.g.
\cite{cappe:moulines:ryden:2005}.

\section{Applications to Bayesian inverse problems in Hidden Markov Models}
\label{sec:MCexperiments}

\subsection{Stochastic volatility model}
Consider the following stochastic volatility model:
\begin{equation*}
X_{t+1} = \phi X_t + \sigma U_{t}\eqsp, \qquad \qquad
Y_t = \beta \rme^{\frac{X_t}{2}} V_t\eqsp,
\end{equation*}
where $X_0\sim\mathcal{N}\left(0, (1-\phi^2)^{-1} \sigma^2\right)$ and
$\{U_t\}_{t\geq 0}$ and $\{V_t\}_{t\geq 0}$ are two sequences of i.i.d.  standard
Gaussian r.v., independent from $X_0$.

We illustrate the convergence of the P-BOEM algorithms and discuss the choice
of some design parameters such as the pair ($\tau_n, N_n$). Data are sampled
using $\phi = 0.95$, $\sigma^{2} = 0.1$ and $\beta^{2} = 0.6$; we estimate
$\param = (\phi, \sigma^2, \beta^2)$ by applying the P-BOEM algorithm and its
averaged version.  All runs are started from $\phi = 0.1$, $\sigma^{2} = 0.6$
and $\beta^{2} = 2$.

Figure~\ref{fig:boxplotSVM} displays the estimation of the three
parameters as a function of the number of observations, over $50$ independent Monte
Carlo runs. The block-size sequence is of the form $\tau_{n}\propto n^{1.2}$. For the SMC step, we choose $N_n = 0.25\cdot \tau_n$; particles are sampled as described in
Algorithm~\ref{alg:FSMC} (see Appendix~\ref{app:alg}) with the bootstrap filter. For each parameter, Figure~\ref{fig:boxplotSVM} displays the empirical median (bold line) and upper and lower quartiles (dotted line).
The averaging procedure is started after $1500$ observations. Both algorithms converge to the true values of the parameters and, once the averaging procedure is started, the variance of the estimation
decreases (estimation of $\phi$ and $\beta^{2}$). The estimation of $\sigma^{2}$ shows that, if the averaging procedure is started with too few observations, the estimation can be slowed down.

 \begin{figure}[!h]
  \centering
  \subfloat[Estimation of $\phi$.]{\includegraphics[width=0.5\textwidth]{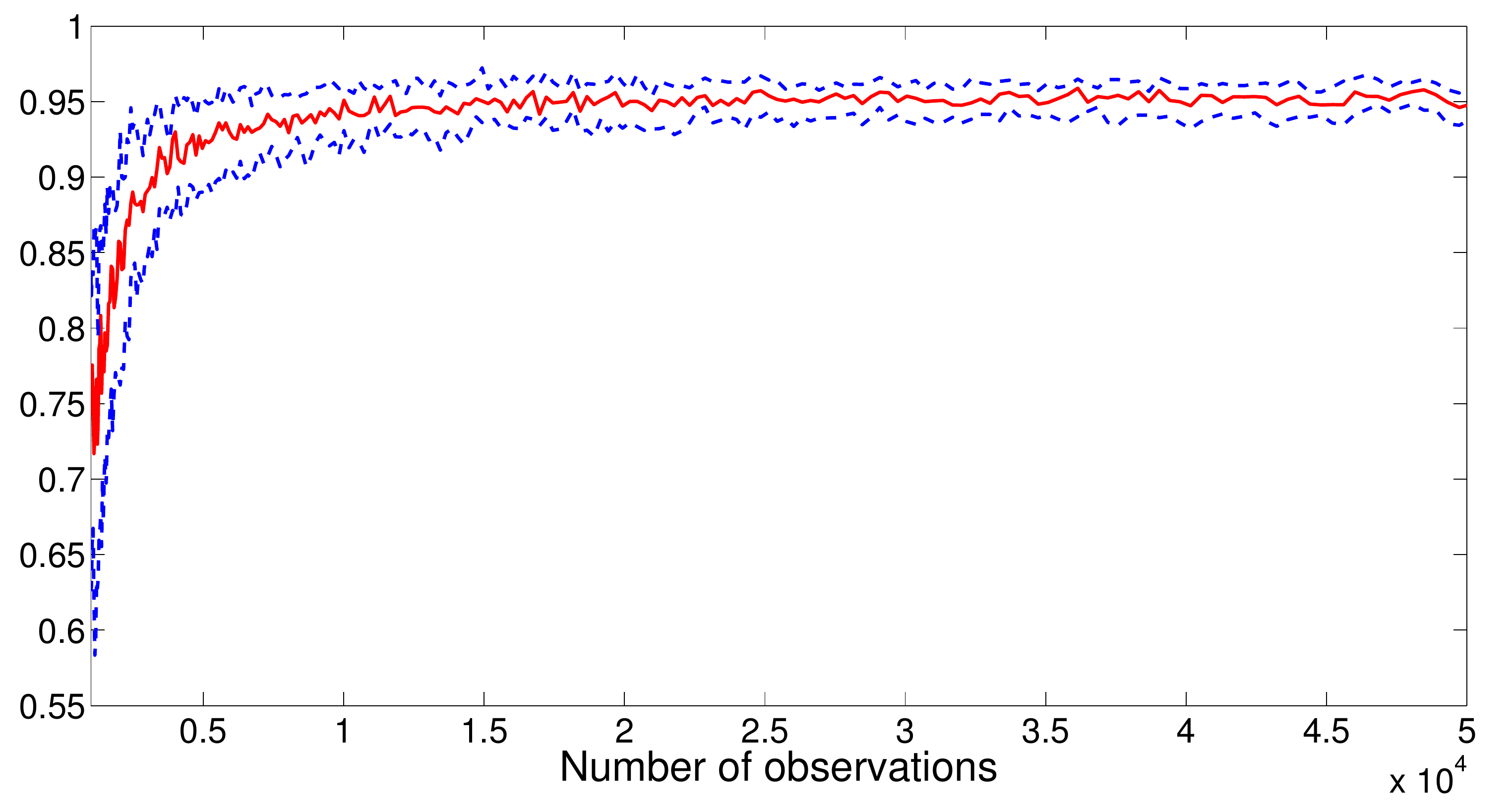}}
  \subfloat[Estimation of $\phi$.]{\includegraphics[width=0.5\textwidth]{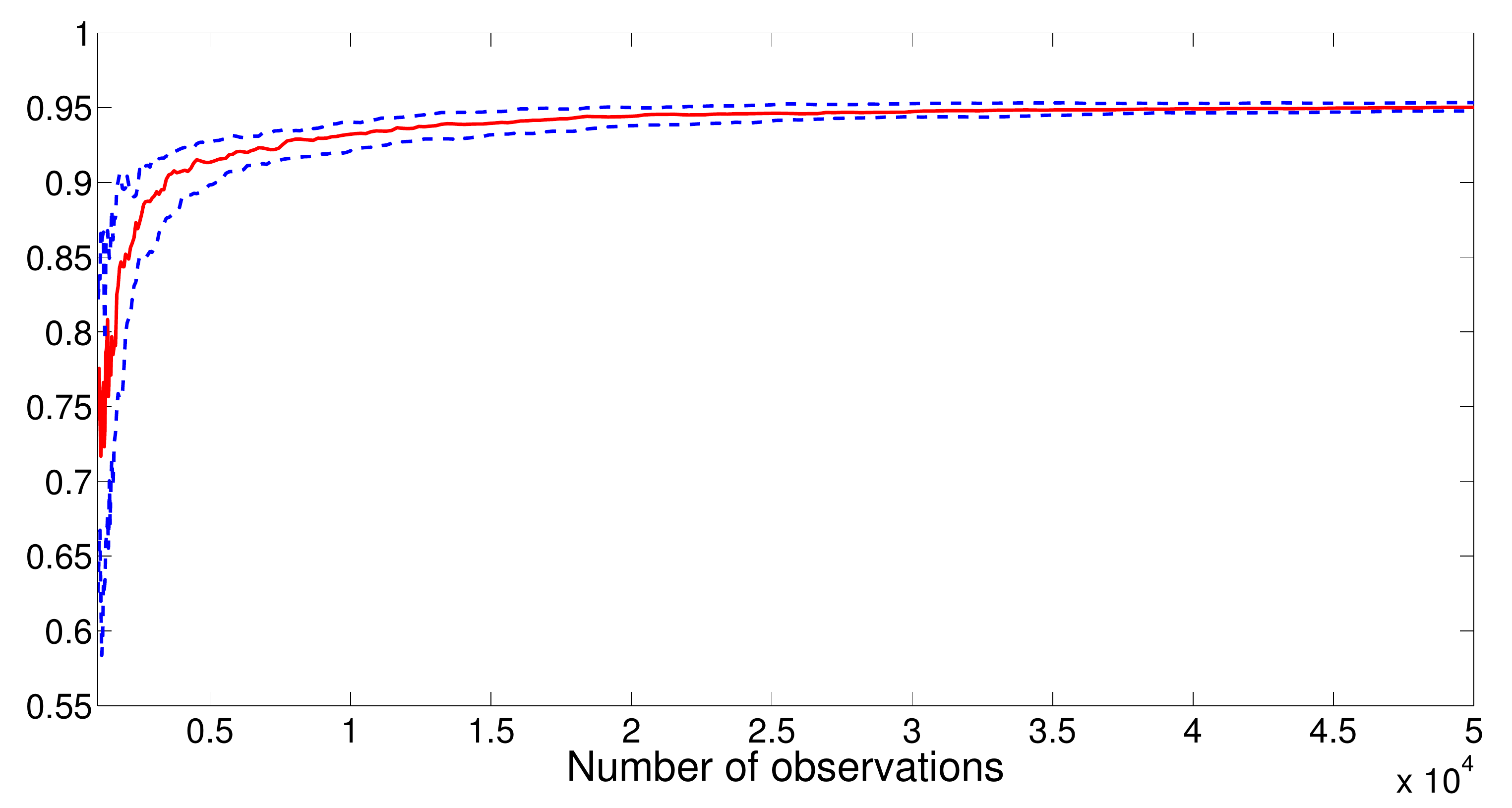}}\\
  \subfloat[Estimation of $\sigma^{2}$.]{\includegraphics[width=0.5\textwidth]{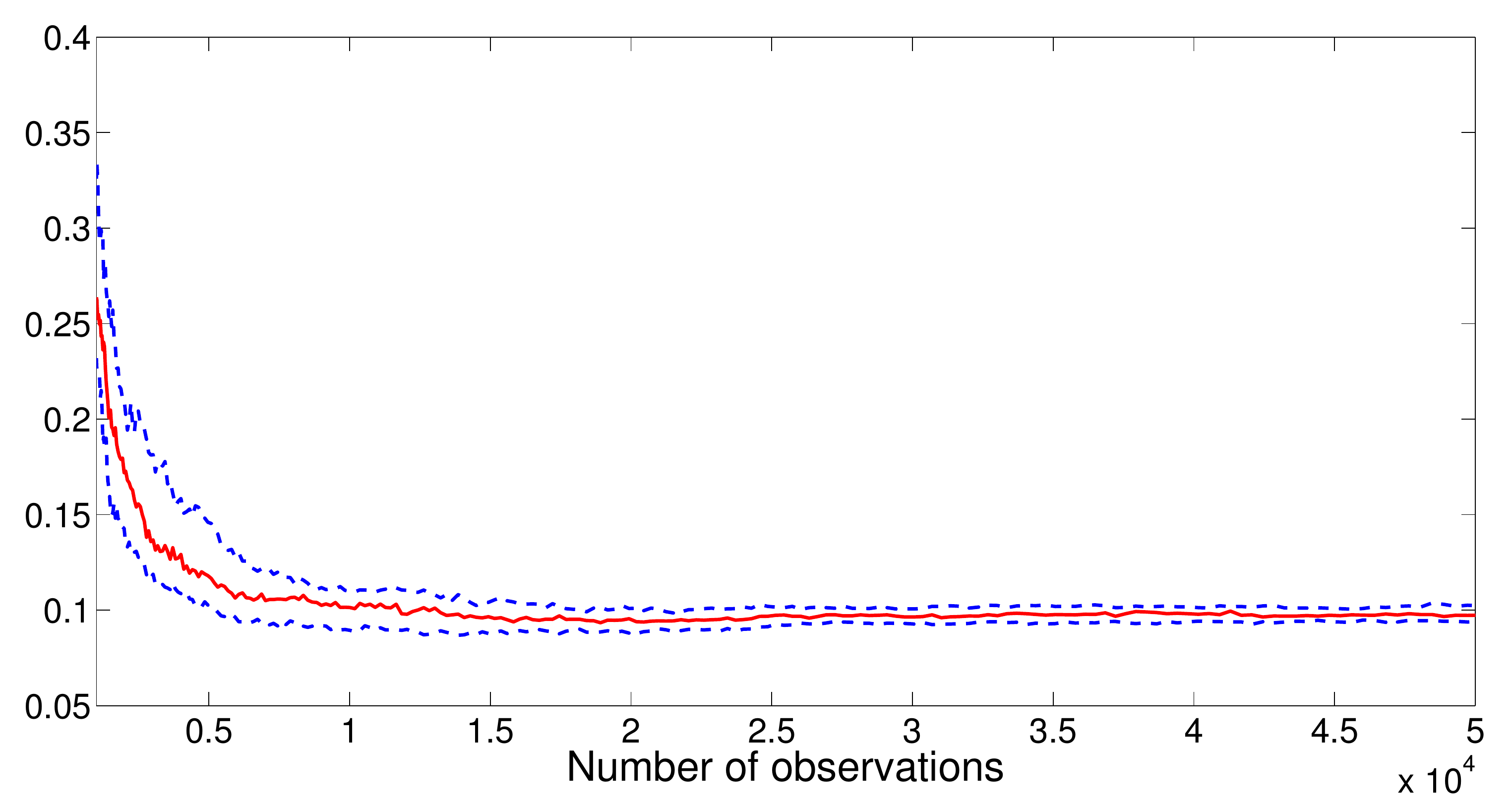}}
  \subfloat[Estimation of $\sigma^{2}$.]{\includegraphics[width=0.5\textwidth]{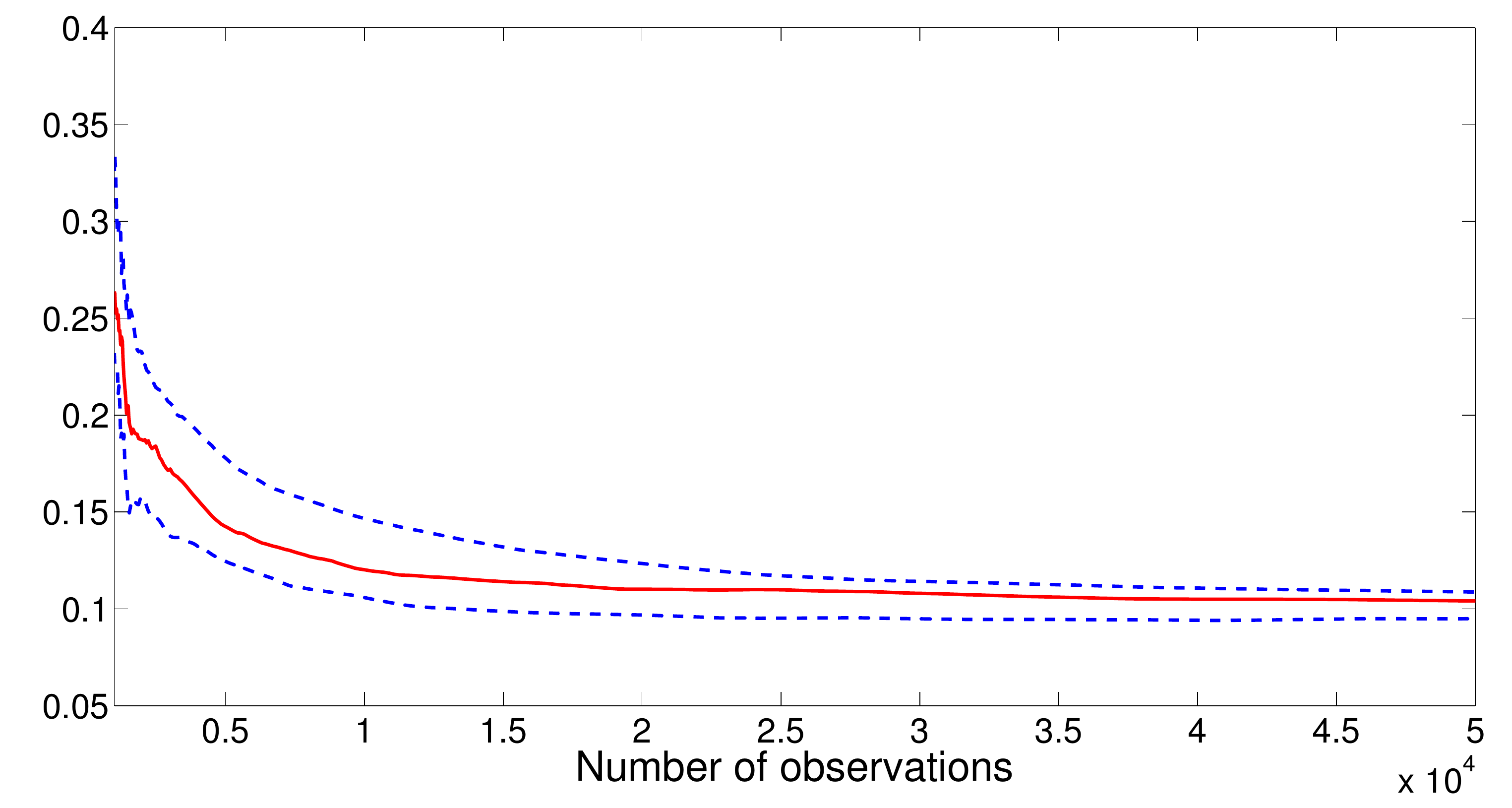}}\\
  \subfloat[Estimation of $\beta^{2}$.]{\includegraphics[width=0.5\textwidth]{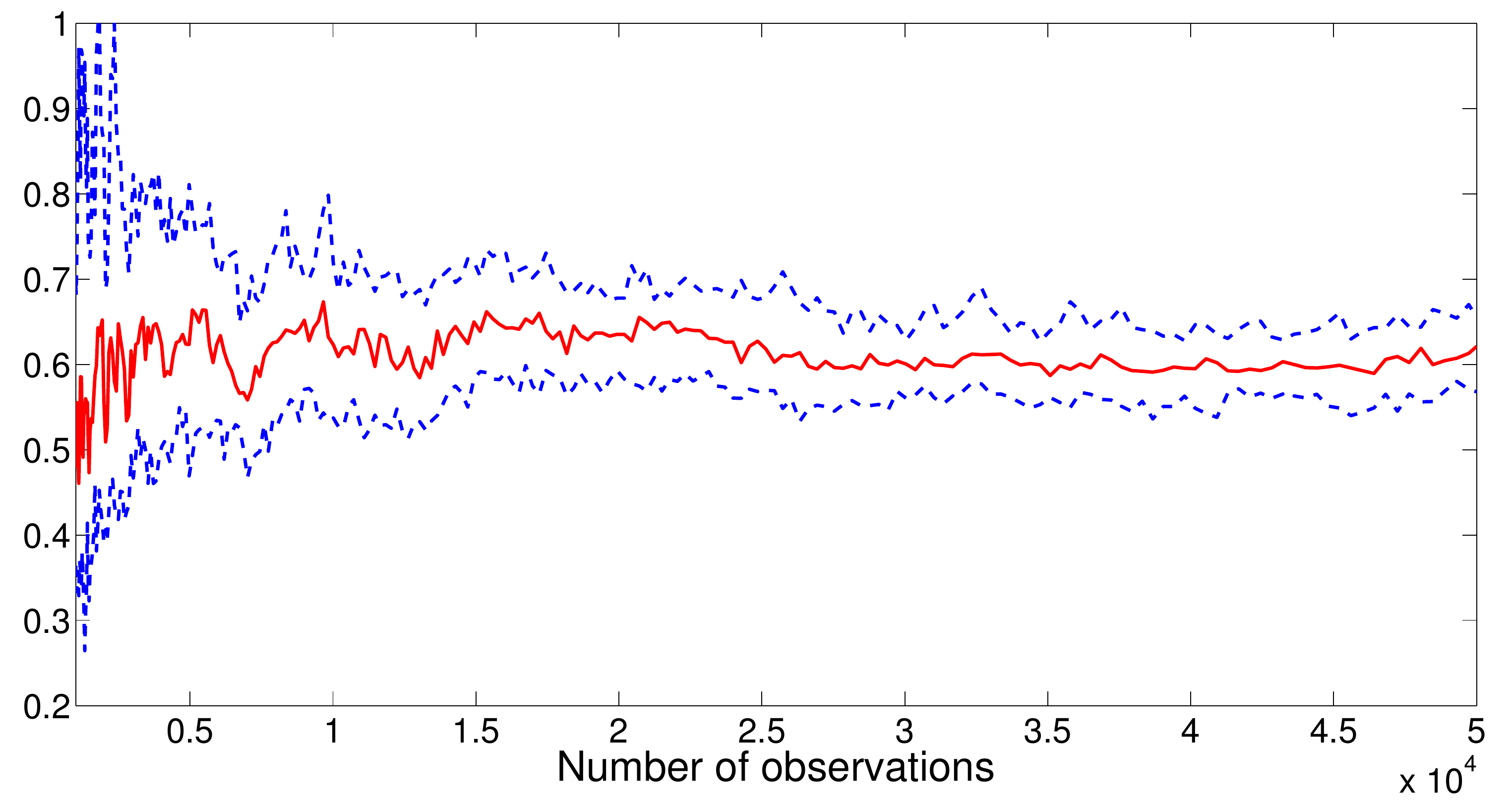}}
  \subfloat[Estimation of $\beta^{2}$.]{\includegraphics[width=0.5\textwidth]{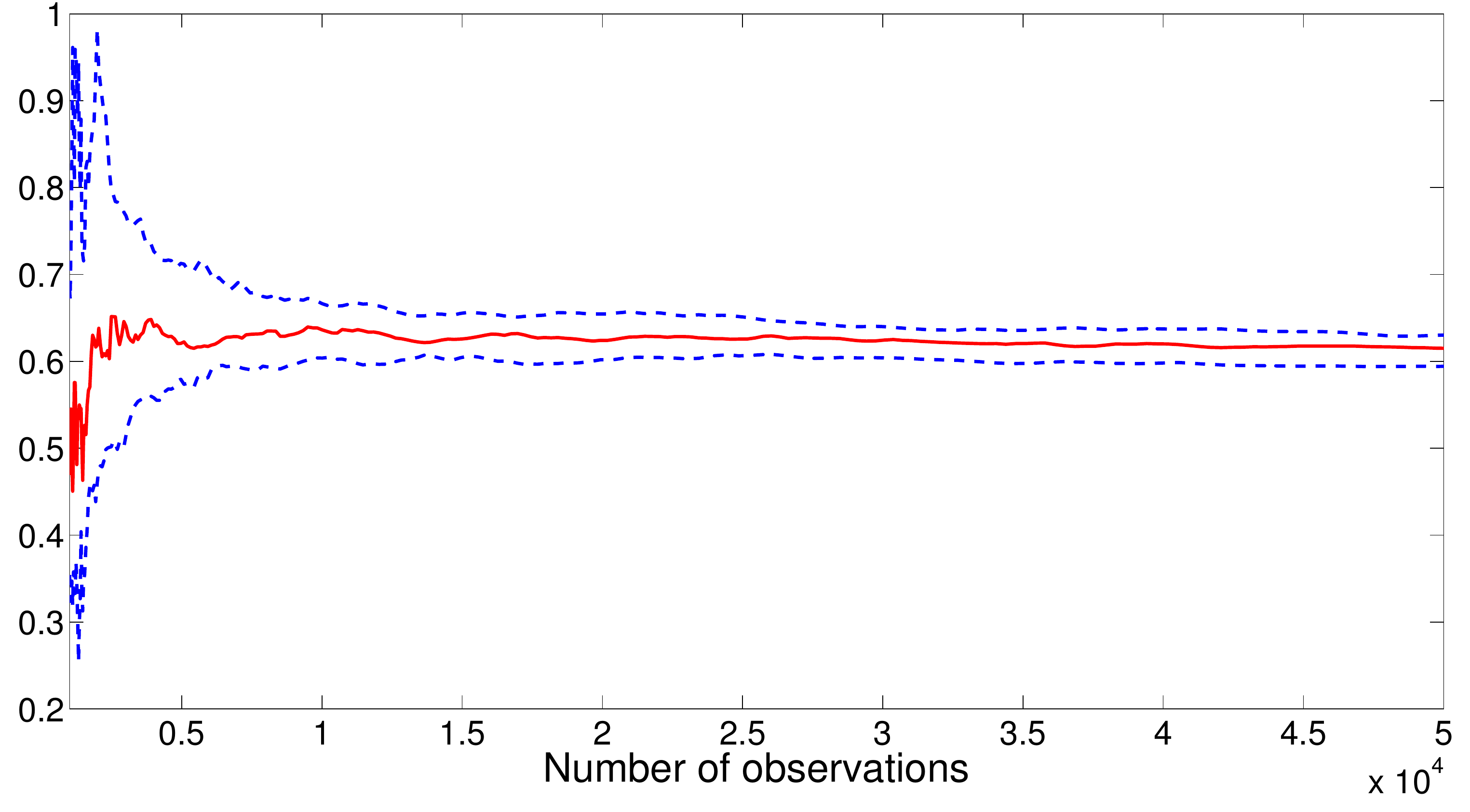}}\\
   \caption{Estimation of $\phi$, $\sigma^{2}$ and $\beta^{2}$ without (left) and with (right) averaging. Each graph represents the empirical median (bold line) and upper and lower quartiles (dotted line) over $50$ independent Monte Carlo runs. The averaging procedure is started after $1500$ observations. The first $1000$ observations are not displayed for better clarity.}
   \label{fig:boxplotSVM}
\end{figure}

We now discuss the role of the pairs $(\tau_n, N_n)$.  Roughly speaking (see
section~\ref{sec:convergence} for a rigorous decomposition), $\tau$ controls
the rate of convergence of $\bar S_{\tau}^{T}(\param, \bfY)$ to $\lim_{\tau \to
  \infty} \bar S_{\tau}^{T}(\param, \bfY) $; and $N$ controls the error between
$\bar S_{\tau}^{T}(\param, \bfY)$ and its SMC approximation.  We will show in Section~\ref{sec:convergence}
that $\lim_n \tau_n = \lim_n N_n = +\infty$ are part of some sufficient
conditions for the P-BOEM algorithms to converge. We thus choose increasing
sequences $\{\tau_n, N_n \}_{n \geq 1}$.  The role of $\tau_n$ has been
illustrated in \cite[Section $3$]{lecorff:fort:2011}. Hence, in this illustration, we fix
$\tau_n$ and discuss the role of $N_n$.  Figure~\ref{fig:varSVM} compares the
algorithms when applied with $\tau_n \propto n^{1.1}$ and $N_{n} =
\sqrt{\tau_{n}}$ or $N_{n}=\tau_{n}$.  The empirical variance (over $50$
independent Monte Carlo runs) of the estimation of $\beta^{2}$ is displayed, as
a function of the number of blocks.  First, Figure~\ref{fig:varSVM} illustrates the  variance decrease provided by the averaged procedure, whatever the block size sequence. Moreover, increasing the number of particles per block improves the variance of the estimation given by the P-BOEM algorithm while the impact on the variance of the averaged estimation is less important. On average, the variance is reduced by a factor of $3.0$ for the P-BOEM algorithm and by a factor of $1.8$ for its averaged version when the number of particles goes from $N_{n}=\sqrt{\tau_{n}}$ to $N_{n}=\tau_{n}$. These practical considerations illustrate the theoretical results derived in Section~\ref{sec:averaging}.

 \begin{figure}[!h]
  \centering
  \subfloat[{\bf P-BOEM}: empirical variance of the estimation of $\beta^{2}$ with $N_{n}=\sqrt{\tau_{n}}$ (dashed line) and $N_{n}=\tau_{n}$ (bold line).]{\includegraphics[width=0.75\textwidth]{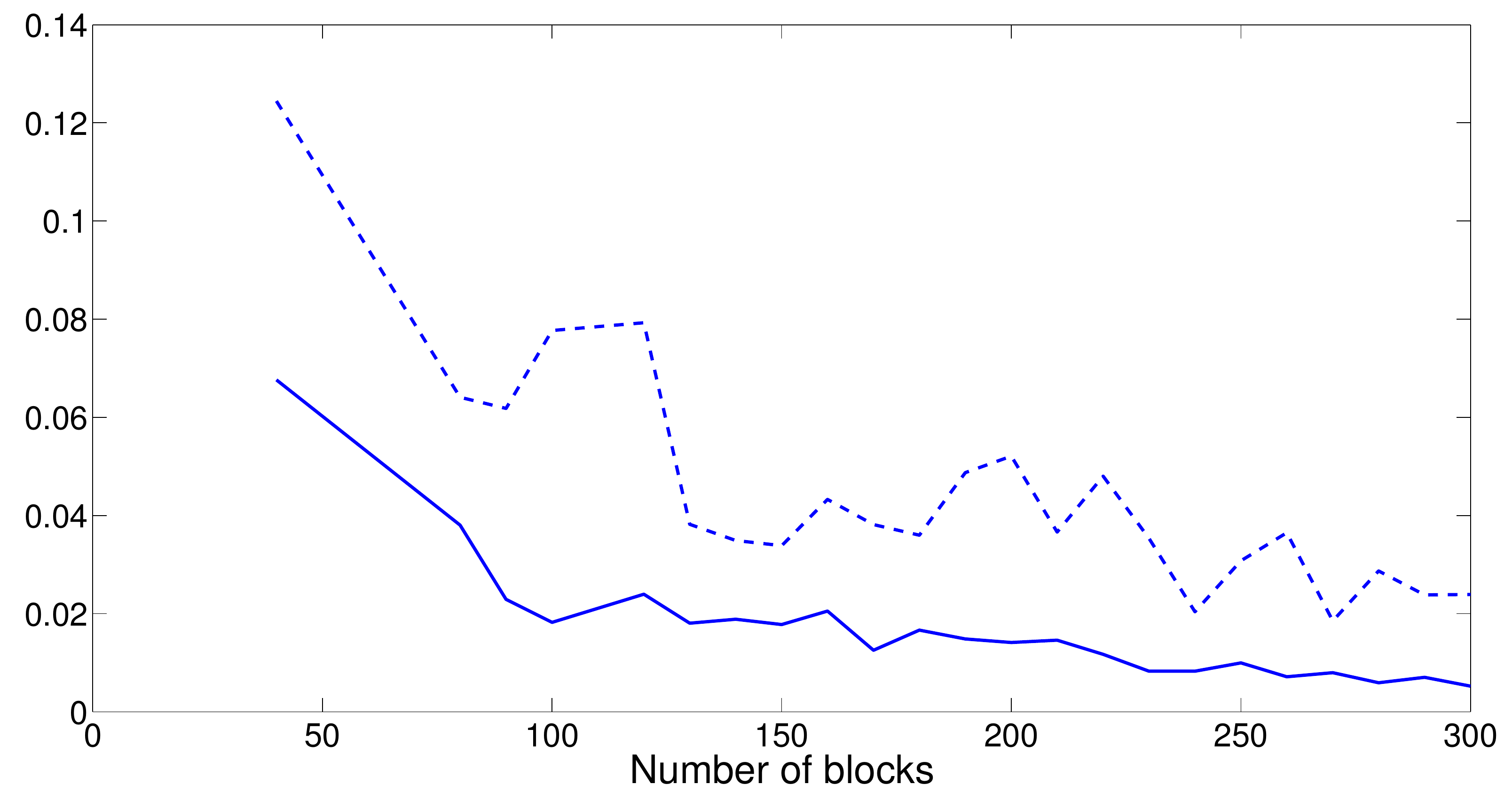}}\\
  \subfloat[{\bf Averaged P-BOEM}: empirical variance of the estimation of $\beta^{2}$ with $N_{n}=\sqrt{\tau_{n}}$ (dashed line) and $N_{n}=\tau_{n}$ (bold line).]{\includegraphics[width=0.75\textwidth]{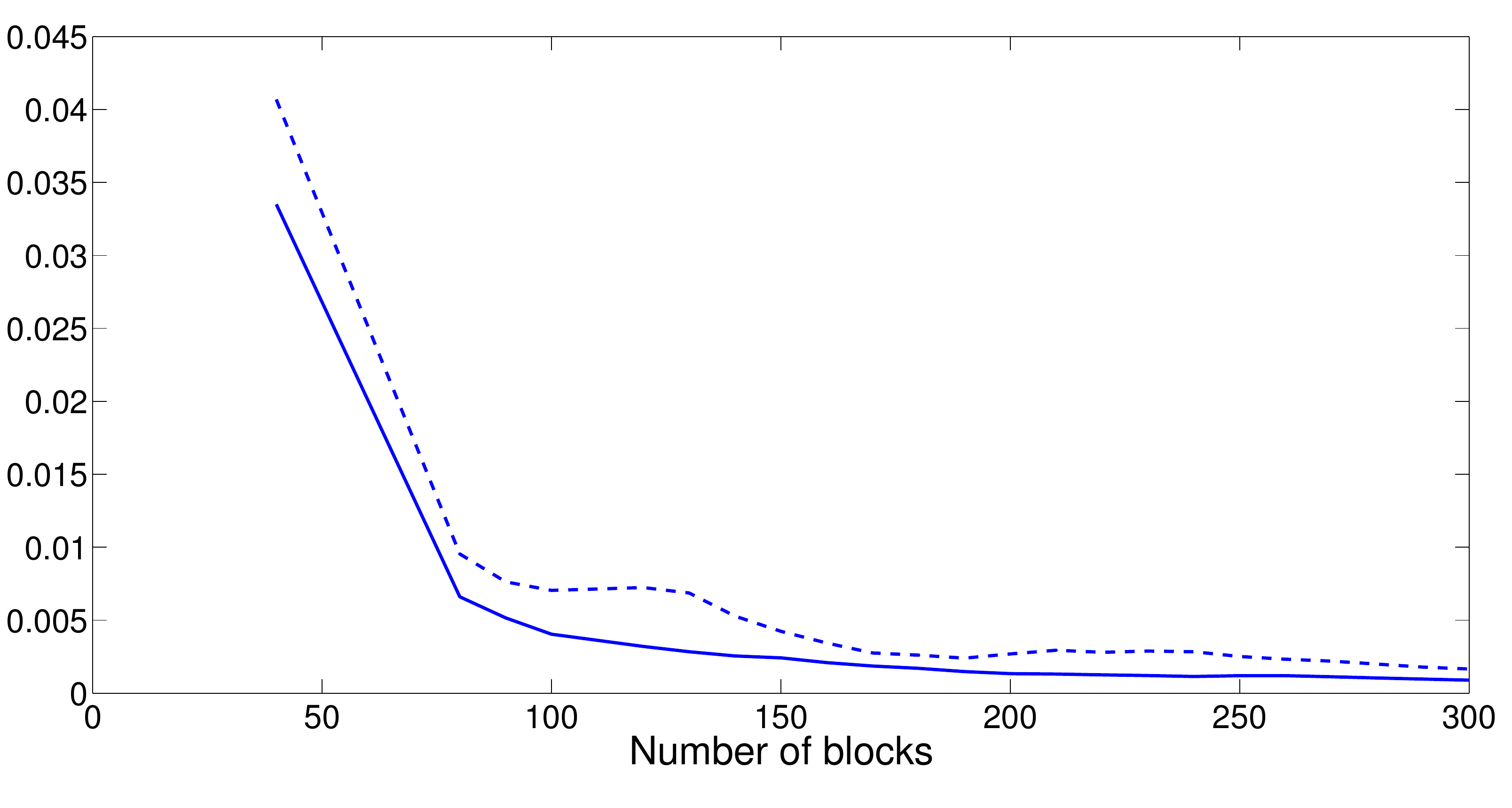}}
   \caption{Empirical variance of the estimation of $\beta^{2}$ with the P-BOEM algorithm (top) and its averaged version (bottom). The averaging procedure is started after the $25$-th block and the variance is displayed after a burn-in time of $35$ blocks.}
   \label{fig:varSVM}
\end{figure} 

Finally, we discuss the role of the initial distribution $\chi$. In all the applications above, we have the same distribution $\chi\equiv\mathcal{N}\left(0, (1-\phi^2)^{-1} \sigma^2\right)$ at the beginning of each block. We could choose a different distribution $\chi_n$ for each block such as, e.g., the filtering distribution at the end of the previous block. We have observed that this particular choice of $\chi_{n}$ leads to the same behavior for both algorithms.

To end this section, the P-BOEM algorithm is compared to the Online EM algorithm outlined in~\cite{cappe:2011} and \cite{delmoral:doucet:singh:2010a}.  These algorithms rely on  a combination of stochastic approximation and SMC methods. According to classical results on stochastic approximation, it is expected that the rate of convergence of the Online EM algorithm behaves like $\gamma_n^{1/2}$, where $\{\gamma_n\}_{n\geq 0}$ is the so called step-size sequence. Hence, $\gamma_{n}$ in the Online EM algorithm is chosen such that $\gamma_n \propto n^{-0.55}$ and the block-size sequence in the P-BOEM algorithm such that  $\tau_n \propto n^{1.2}$.  The number of particles used in the Online EM algorithm is fixed and chosen so that the computational costs of both algorithms are similar. Provided that  $N_{n}\propto \tau_{n}$ in the P-BOEM algorithm, this leads to a choice of $70$ particles for the Online EM algorithm.  We report in Figure~\ref{fig:boxplotSVM_BOEM-OEM}, the estimation of $\phi$ and  $\sigma^{2}$ for a Polyak-Ruppert  averaged Online EM algorithm (see \cite{polyak:1990}) and the averaged P-BOEM algorithm as a function of the number of observations.  The averaging procedure is started after about $1500$ observations.  As noted in \cite[Section $3$]{lecorff:fort:2011} for a constant sequence $\{N_{n}\}_{n\geq 0}$ this figure shows that both algorithms
behave similarly. For the estimation of $\phi$ and $\beta^{2}$, the variance is smaller for the P-BOEM algorithm and the convergence is faster for the P-BOEM algorithm in the case of $\beta^{2}$. Conclusions are different for the estimation of $\sigma^{2}$: the variance is smaller for the P-BOEM  algorithm but the Online EM algorithm converges a bit faster. The main advantage of the P-BOEM algorithm is that it relies on approximations which can be controlled in such a way that we are able to show that the limiting points of the P-BOEM algorithms are the stationary points of the limiting normalized log-likelihood of the observations. 

 \begin{figure}[!h]
  \centering
  \subfloat[Estimation of $\phi$.]{\includegraphics[width=0.6\textwidth]{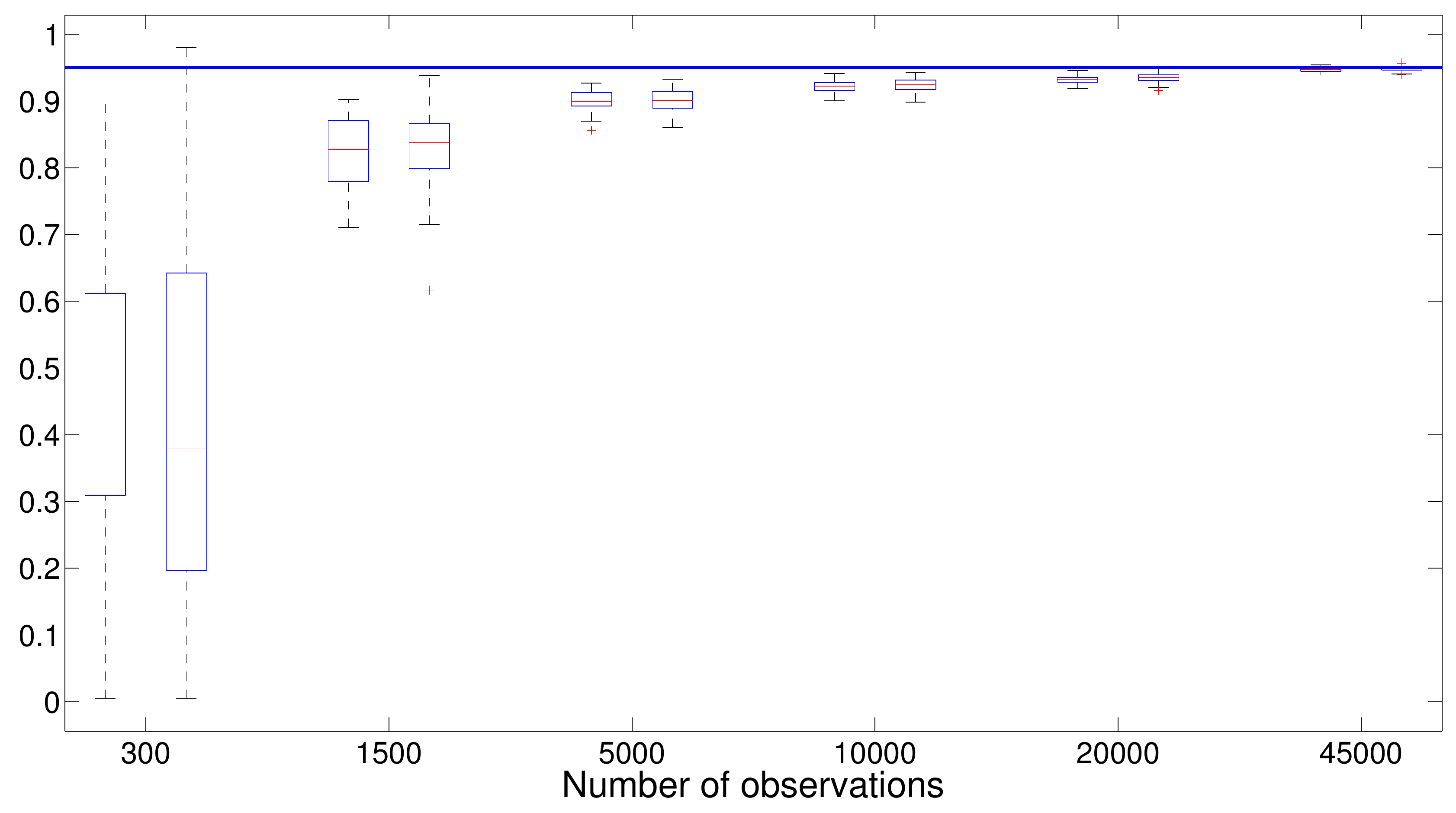}}\\
  \subfloat[Estimation of $\sigma^{2}$.]{\includegraphics[width=0.6\textwidth]{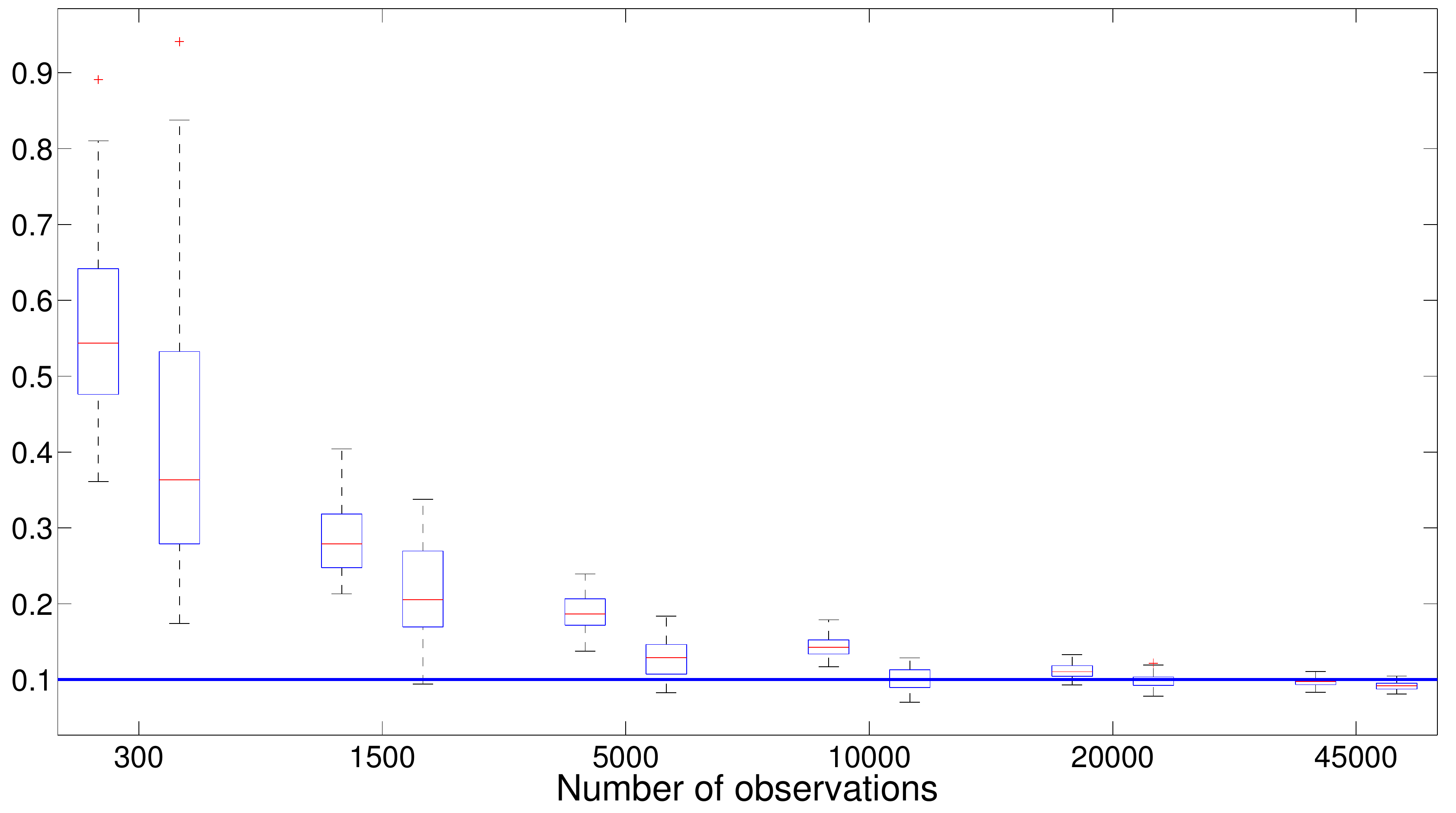}}\\
  \subfloat[Estimation of $\beta^{2}$.]{\includegraphics[width=0.6\textwidth]{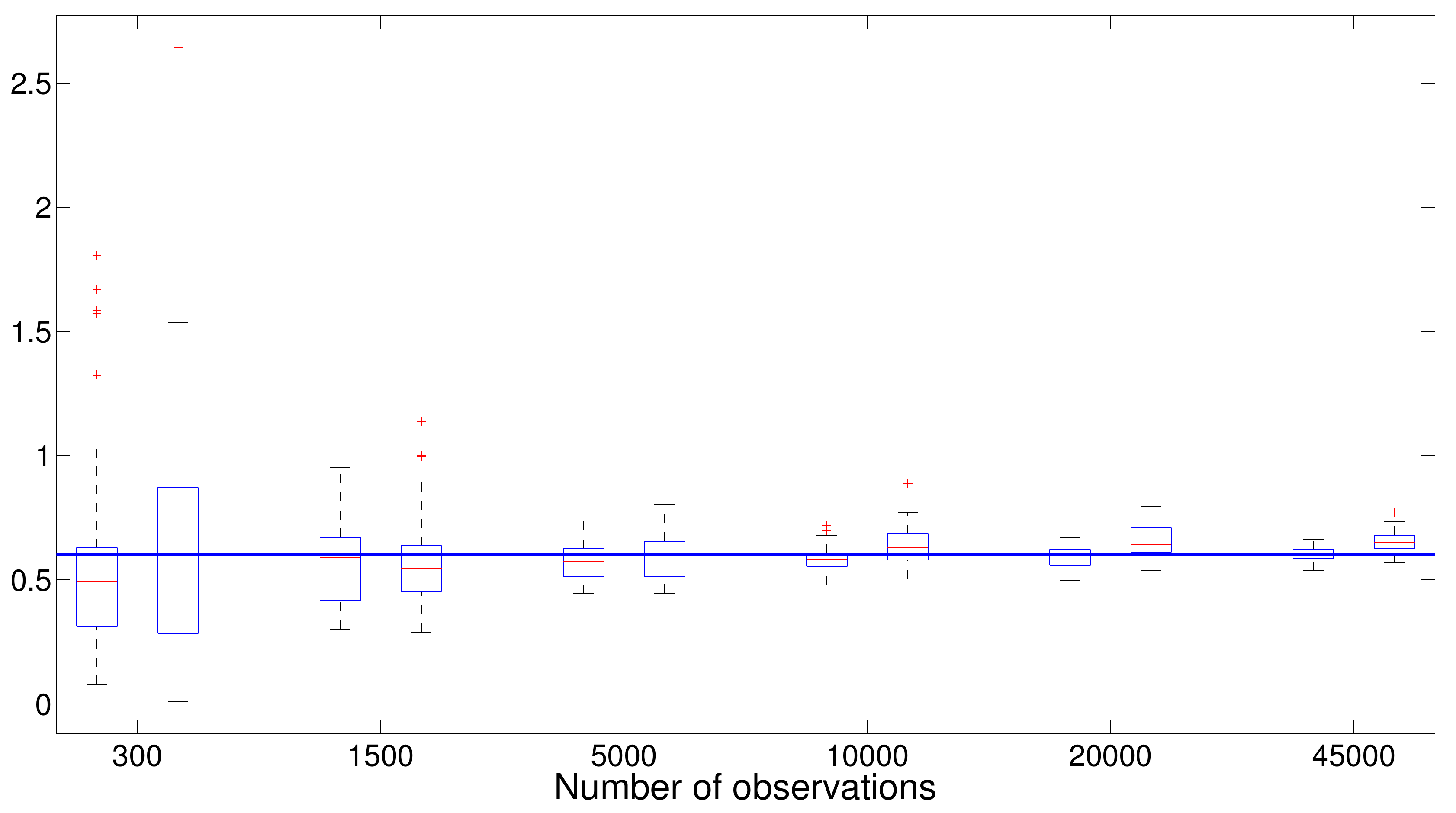}}
   \caption{Estimation of $\phi$, $\sigma^{2}$ and $\beta^{2}$ with the averaged P-BOEM algorithm (left)  and a Polyak-Ruppert averaged version of the Online EM algorithm (right) after $300, 1500, 5000, 10000, 20000$ and $45000$ observations. The averaging procedure is started after about $1000$ observations (which corresponds to the $25$-th block for the P-BOEM algorithm).}
   \label{fig:boxplotSVM_BOEM-OEM}
\end{figure}
\subsection{Simultaneous Localization And Mapping}
The Simultaneous Localization And Mapping (SLAM) problem arises when a mobile
device wants to build a map of an unknown environment and, at the same time,
has to estimate its position in this map.  The common statistical approach for
the SLAM problem is to introduce a state-space model.  Many solutions have been proposed
depending on the assumptions made on the transition and observation models, and
on the map (see e.g.
\cite{burgard:fox:thrun:2006}, \cite{martinezcantin:2008} and \cite{montemerlo:2003}).  In
\cite{martinezcantin:2008} and \cite{lecorff:fort:moulines:2011}, it is proposed to see
the SLAM as an inference problem in HMM: the localization of the robot is the
hidden state with Markovian dynamic, and the map is seen as an unknown
parameter. Therefore, the mapping problem is answered by solving the inference
task, and the localization problem is answered by approximating the
conditional distribution of the hidden states given the observations.

In this application, we consider a statistical model for a landmark-based SLAM
problem for a bicycle manoeuvring on a plane surface.

Let $x_t \eqdef \{x_{t,i}\}_{i=1}^{3}$ be the robot position, where $x_{t,1}$ and
$x_{t,2}$ are the robot's cartesian coordinates and $x_{t,3}$ its orientation.
At each time step, deterministic controls are sent to the robot so that it
explores a given part of the environment. Controls are denoted by $(
v_t,\psi_t)$ where $\psi_t$ stands for the robot's heading direction and $v_t$
its velocity. The robot position at time $t$, given its previous position at time $t-1$
and the noisy controls $(\hat{v}_{t},\hat{\psi}_{t})$, can be written as
\begin{equation}\label{eq-transitionmodel}
x_t   =  f(x_{t-1},\hat{v}_{t},\hat{\psi}_{t})\eqsp,
\end{equation}
where $(\hat{v}_{t},\hat{\psi}_{t})$ is a $2$-dimensional Gaussian distribution
with mean $( v_t,\psi_t)$ and known covariance matrix $Q$. In this
contribution we use the kinematic model of the front wheel of a bicycle (see
e.g.  \cite{bailey:2006}) where the function $f$ in \eqref{eq-transitionmodel}
is given by
\[
f(x_{t-1},\hat{v}_{t},\hat{\psi}_{t}) = x_{t-1} + \begin{pmatrix} \hat{v}_{t}d_{t}\cos(x_{t-1,3} + \hat{\psi}_{t})\\\ \hat{v}_{t}d_{t}\sin(x_{t-1,3} + \hat{\psi}_{t})\\ \hat{v}_{t}d_{t}B^{-1}\sin(\hat{\psi}_{t})\end{pmatrix}\eqsp,
\]
where $d_{t}$ is the time period between two successive positions and $B$ is the robot wheelbase.

The $2$-dimensional environment is represented by a set of landmarks $\theta
\eqdef \{\theta_{j}\}_{ 1 \leq j \leq q}$, $\theta_{j} \in \mathbb{C}$ being
the position of the $j-$th landmark.  The total number of landmarks $q$ and the
association between observations and landmarks are assumed to be known.

At time $t$, the robot observes the distance and the angular position of all
landmarks in its neighborhood; let $c_t \subseteq \{1, \cdots, q\}$ be the set
of observed landmarks at time $t$. It is assumed that the observations
$\{y_{t,i}\}_{i \in c_t}$ are independent and satisfy
\begin{equation*}
y_{t,i} = h(x_t,\theta_{i})+\delta_{t,i}\eqsp,
\end{equation*} 
where $h$ is defined by
\begin{equation*}
  h(x,\boldsymbol{\kappa}) \eqdef \begin{pmatrix}\sqrt{(\boldsymbol{\kappa}_1-x_1)^2 + (\boldsymbol{\kappa}_2-x_2)^2
  }\\\arctan{\frac{\boldsymbol{\kappa}_2-x_2}{\boldsymbol{\kappa}_1-x_1}}-x_3\end{pmatrix}\eqsp,
\end{equation*}
and the noise vectors $\{\delta_{t,i}\}_{t,i}$ are i.i.d Gaussian
$\mathcal{N}\left(0,R\right)$. $R$ is assumed to be known.

The model presented in this Section does not take into account all the issues arising in the SLAM problem (such as the association process which is assumed to be known and the known covariance matrices). The aim is to prove that the BOEM algorithm and its averaged version have satisfying behavior even in the challenging framework described above. The observation and motion models are highly nonlinear and we show that the BOEM algorithm remains stable in this experiment. Several solutions have been proposed to solve the association problem (see e.g. \cite{burgard:fox:thrun:2006} for a solution based on the likelihood of the observations) and could be adapted to our case.
We want to estimate $\param = \{\param_j\}_{j =1}^q$ by applying the P-BOEM
algorithms. In this paper, we use simulated data. $q=15$ landmarks
are drawn in a square of size $45m\mathrm{x}45m$. The robot path is sampled with a
given set of controls. Using the {\em true} positions of all landmarks in the
map and the true path of the robot (see the dots and the
bold line on Figure~\ref{fig:BOEMSLAM}), observations are sampled by setting: $R = \begin{pmatrix}\sigma_r^2 & \rho\\
  \rho & \sigma_b^2\end{pmatrix}\eqsp,$ where $\sigma_r = 0.5\mbox{m}$,
$\sigma_b = \frac{\pi}{60}\mbox{rad}$ and $\rho = 0.01$. We choose $Q =
\mbox{diag}(\sigma_v^2, \sigma_{\phi}^2)$ where $\sigma_v = 0.5\mbox{m.s}^{-1}$, $\sigma_{\psi} = \frac{\pi}{60}\mbox{rad}$ and $B=1.5\mathrm{m}$.

In this model, the transition denoted by $m_{\param}$ does not depend on the
map $\param$ (see \eqref{eq-transitionmodel}) and the marginal likelihood
$g_{\param}$ is such that the complete data likelihood does not belong to the curved exponential family:
\begin{equation}
  \label{eq:expression:logg}
 \sum_{i\in c_{t}} \ln g_{\param}(x_{t},y_{t,i})\propto \sum_{i\in c_{t}}
\left[y_{t,i}-h(x_{t},\param_{i})\right]^{\star}R^{-1}\left[y_{t,i}-h(x_{t},\param_{i})\right]
\eqsp.
\end{equation}
Hence, in order to apply Algorithm~\ref{alg:PBOEM}, at the beginning of each
block, $g_{\param}$ is approximated by a function depending on the current
parameter estimate so that the resulting approximated model belongs to the
curved exponential family (see \cite{lecorff:fort:moulines:2011}).  As can be seen
from (\ref{eq:expression:logg}), approximating the function $\boldsymbol{\kappa}\mapsto
h(x,\boldsymbol{\kappa})$ by its first-order Taylor expansion at $\param_{i}$ leads to a
quadratic approximation of $g_{\param}$. This approach is commonly used in the
SLAM framework to use the properties of linear Gaussian models (see e.g.
\cite{burgard:fox:thrun:2006}). 

As the landmarks are not observed all the time, we choose a slowly increasing
sequence $\{\tau_{n}\propto n^{1.1}\}_{n\geq 1}$ so that the number of updates
is not too small (in this experiment, we have $60$ updates for a total number
of observations of $2000$). As the total number of observations is not so large
(the largest block is of length $60$), the number of particles is chosen to be
constant on each block: for all $n\geq 1$, $N_{n} = 50$. For the SMC step, we
apply Algorithm~\ref{alg:FSMC} with the {\em bootstrap filter}.

For each run the estimated path (equal to the weighted mean of the particles)
and the estimated map at the end of the loop ($T=2000$) are stored.
Figure~\ref{fig:BOEMSLAM} represents the mean estimated path and the mean map
over $50$ independent Monte Carlo runs. It highlights the good performance of
the P-BOEM algorithm in a more complex framework.

 \begin{figure}[!h]
  \centering
  \includegraphics[width=0.6\textwidth]{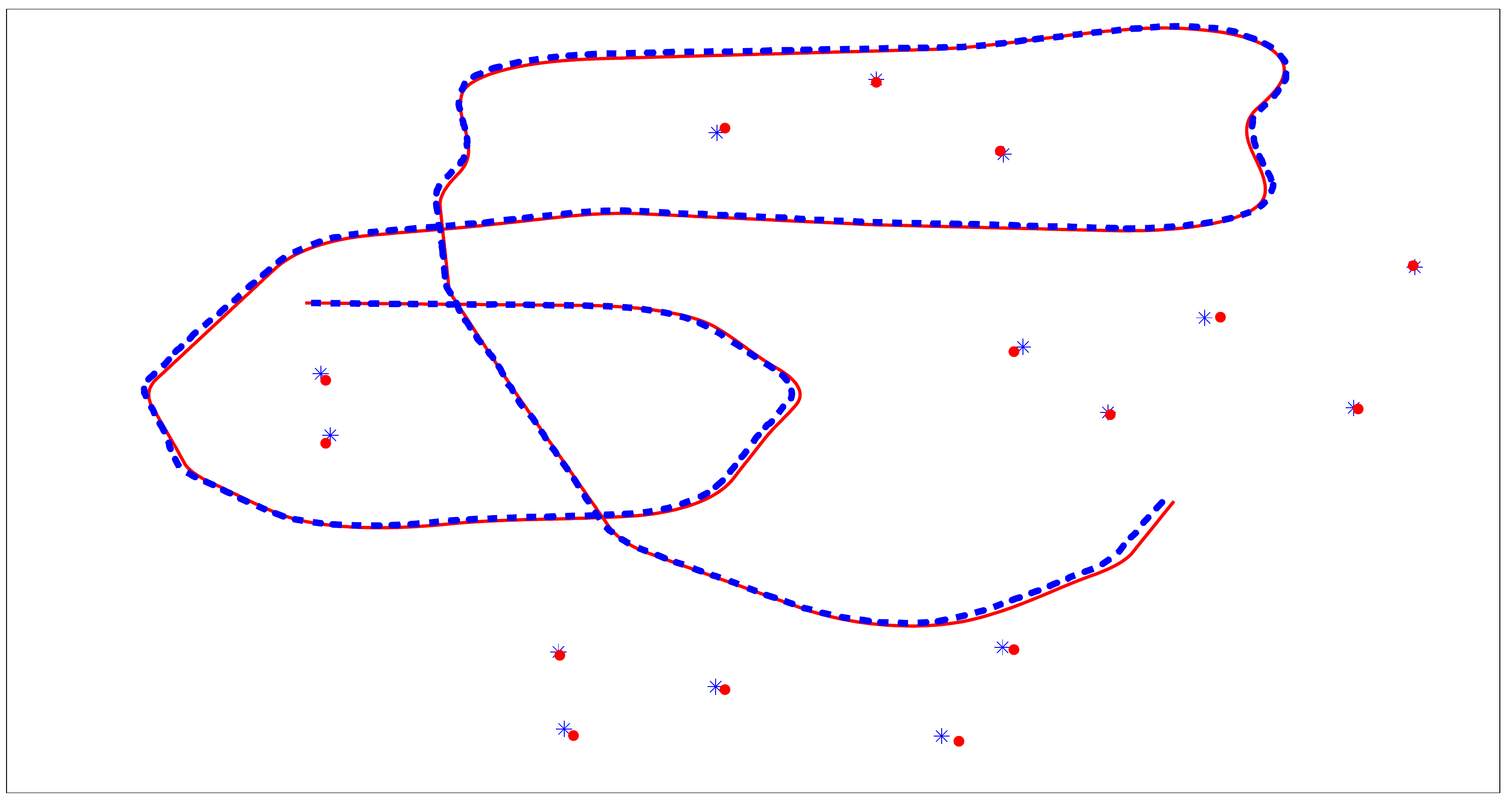}
   \caption{True trajectory (bold line) and true landmark positions (balls) with the estimated path (dotted line) and the landmarks' estimated positions (stars) at the end of the run ($T=2000$).}
   \label{fig:BOEMSLAM}
\end{figure} 

We also compare our algorithm to the {\em marginal SLAM} algorithm proposed by
\cite{martinezcantin:2008}. In this algorithm, the map is also modeled as a
parameter to learn in a HMM model; SMC methods are used to estimate the map in
the maximum likelihood sense.  The Marginal SLAM algorithm is a gradient-based approach for solving the recursive maximum likelihood procedure. Note that, in the case of i.i.d. observations,
\cite{titterington:1984} proposed to update the parameter estimate each time a new
observation is available using a stochastic gradient approach. Figure~\ref{fig:BOEM_Marginal_SLAM} illustrates the estimation of the position
of each landmark. The P-BOEM algorithm is applied using the same parameters as above and the
{\em marginal SLAM} algorithm uses a sequence of step-size $\{\gamma_{n}\propto
n^{-0.6}\}_{n\geq 1}$. We use the averaged version of the P-BOEM algorithm and a
Polyak-Ruppert based averaging procedure for the {\em marginal SLAM} algorithm (see
\cite{polyak:1990}).  For each landmark the last estimation (at the end of the
loop) of the position is stored for each of the $50$ independent Monte Carlo
runs.  Figure~\ref{fig:BOEM_Marginal_SLAM} displays the distance between the
estimated position and the true position for each landmark. In this experiment,
the P-BOEM based SLAM algorithm outperforms the {\em marginal SLAM} algorithm.

\begin{figure}[!h]
  \centering
  \includegraphics[width=0.6\textwidth]{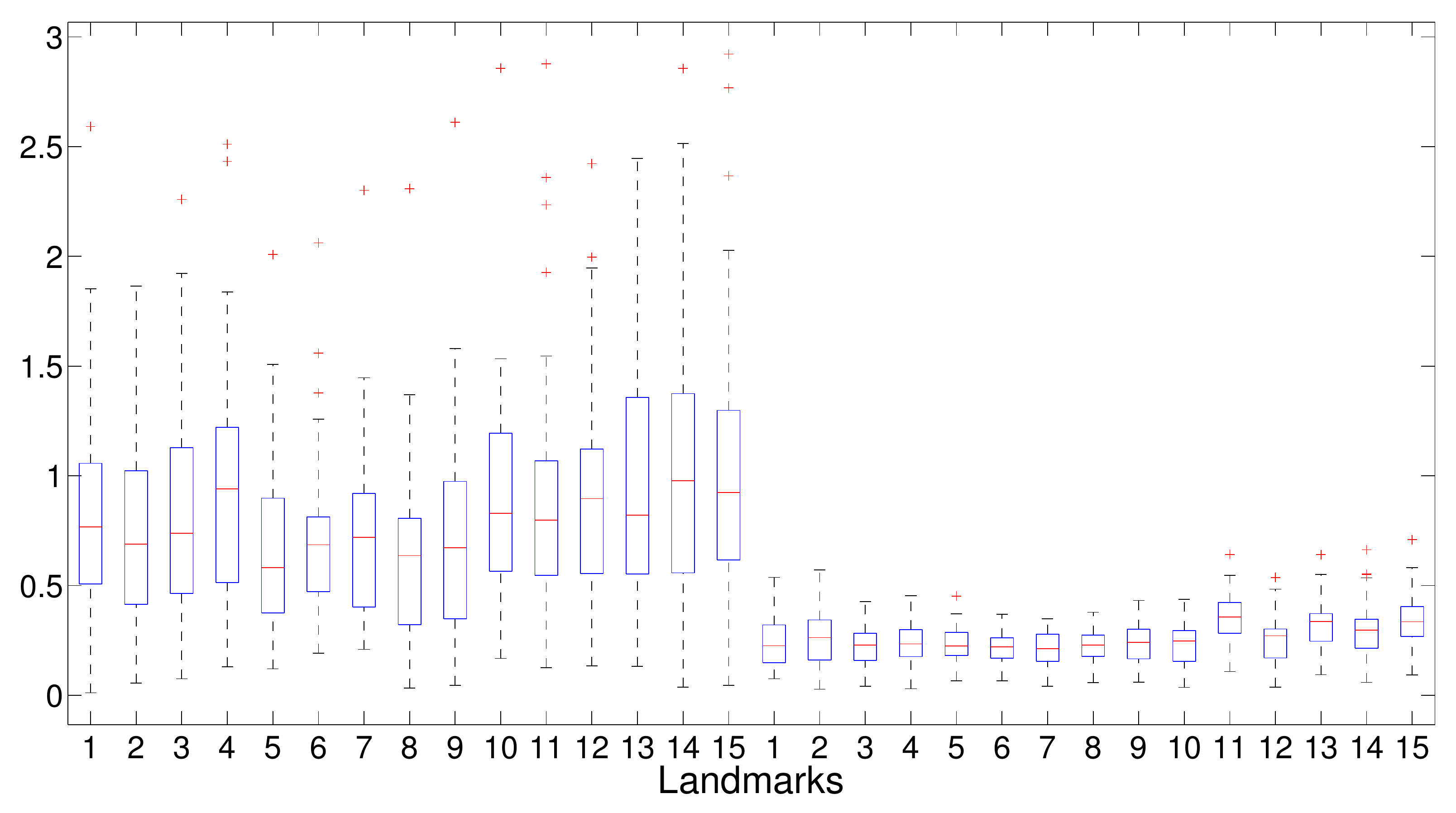}
   \caption{Distance between the final estimation and the true position for each of the $15$ landmarks with the averaged {\em marginal SLAM} algorithm (left) and the averaged P-BOEM algorithm (right).}
   \label{fig:BOEM_Marginal_SLAM}
\end{figure} 

\section{Convergence of the Particle Block Online EM algorithms}
\label{sec:convergence}
In this section, we analyze the limiting points of the P-BOEM algorithm.  We
prove in Theorem~\ref{th:Pbonem:conv} that the P-BOEM algorithm has the same limit points as
a so-called {\em limiting EM} algorithm, which defines a
sequence $\{\param_n\}_{n \geq 0}$ by $\param_{n+1} = \bar \param \left[\bar
S(\param_n)\right]$ where $\bar S(\param)$ is the a.s. limit $\lim_{\tau \to +\infty}
\bar S_{\tau}^{T}(\param, \bfY)$ (defined by (\ref{eq:rewrite:barS})). As
discussed in \cite[Section 4.3.]{lecorff:fort:2011}, the set of limit
points of the limiting EM algorithm is the set of stationary points of the contrast
function $\ell(\param)$, defined as the a.s. limit of the normalized
log-likelihood of the observations, when $T\to +\infty$.
The convergence result below on the P-BOEM algorithm requires two sets of assumptions: conditions
A\ref{assum:strong} to A\ref{assum:block-size} are the same as in \cite{lecorff:fort:2011} and
imply the convergence of the BOEM algorithm; assumptions A\ref{assum:dub:lec} and
A\ref{assum:moment:w} are introduced to control the Monte Carlo error.

\subsection{Assumptions}
Consider the following assumptions
\begin{hypA}\label{assum:strong}
  There exist $\sigma_{-}$ and $\sigma_{+}$ s.t. for any
  $\left(x,x^{\prime}\right)\in\Xset^2$ and any $\param \in \paramset$, $0
  <\sigma_{-} \leq m_\param(x,x^{\prime})\leq \sigma_{+}$.  Set $\rho \eqdef 1 -
  (\sigma_-/\sigma_+)\eqsp.$
\end{hypA}

Define, for all $y\in\Yset$,
\begin{equation}
      \label{eq:b+:b-}
      b_-(y)  \eqdef  \inf_{\param \in \paramset} \int g_\param(x,y) \lambda(\rmd x)  \quad\mbox{and}\quad  b_+(y) \eqdef  \sup_{\param \in \paramset} \int g_\param(x,y) \lambda(\rmd x) \eqsp.
    \end{equation}
    
For any sequence of r.v.  $Z\eqdef\{Z_{t}\}_{t\in\Zset}$ on
$(\Omega,\widetilde{\PP},\mathcal{F})$, let
\begin{equation}
\label{eq:Zfield}
\mathcal{F}_{k}^{Z} \eqdef \sigma\left(\{Z_{u}\}_{u\leq k}\right)\quad \mbox{and}\quad \mathcal{G}_{k}^{Z} \eqdef \sigma\left(\{Z_{u}\}_{u\geq k}\right)
\end{equation}
be $\sigma$-fields associated to $Z$. We also define the mixing coefficients
by, see \cite{davidson:1994},
\begin{equation}
\label{eq:def:mixing}
\mix^{Z}(n) \eqdef \underset{u\in\Zset}{\sup}\,\mix^{}(\mathcal{G}_{u+n}^{Z},\mathcal{F}_{u}^{Z})\eqsp, \forall\; n\geq0\eqsp,
\end{equation} 
where for any $\sigma$-algebras $\mathcal{F}$ and $\mathcal{G}$,
\begin{equation}
\label{eq:def:mixing:2}
\mix^{}(\mathcal{G},\mathcal{F}) \eqdef \underset{B\in\mathcal{G}}{\sup}\,|\widetilde{\PP}(B| \mathcal{F}) - \widetilde{\PP}(B)|\eqsp.
\end{equation} 
For $p>0$ and $Z$ a $\Rset^{d}$-valued random variable measurable w.r.t. the $\sigma$-algebra
$\mathcal{F}$, set 
\[
\lpnorm{Z}{p}\eqdef\left(\CE{}{}{ |Z|^p}\right)^{1/p} \eqsp.  
\]
\begin{hypA}\label{assum:moment:sup} 
\hspace{-0.15cm}-($p$)  
 $\lpnorm{\sup_{x,x' \in \Xset^2} \, |S(x,x',\bfY_{0})|}{p}<+\infty\eqsp.$
\end{hypA}

\begin{hypA}\label{assum:obs}
\begin{enumerate}[(a)]
\item \label{assum:obs:erg} $\bfY$ is a stationary sequence such that there exist $C\in[0,1)$ and $\mix\in(0,1)$ satisfying, for any $n\geq 0$,
  $\mix^{\bfY}(n) \leq C\mix^{n}$, where $\mix^{\bfY}$ is defined in
  \eqref{eq:def:mixing}.
\item \label{assum:obs:b+:b-} $\CE[]{}{}{|\log b_-(\bfY_0)| +|\log b_+(\bfY_0)| }<+\infty$.
\end{enumerate}
\end{hypA}

\begin{hypA}\label{assum:block-size} 
There exist $c>0$ and $a>1$ such that for all $n\ge 1$, $\tau_{n} = \lfloor cn^{a}\rfloor$.
\end{hypA}

Assumptions A\ref{assum:strong} to A\ref{assum:block-size} are the same as
in~\cite{lecorff:fort:2011}.  A\ref{assum:strong}, referred to as the {\em
  strong mixing condition}, is used to prove the uniform forgetting property of
the initial condition of the filter, see e.g.
\cite{delmoral:guionnet:1998} and \cite{delmoral:ledoux:miclo:2003}. 
This assumption is easy to check in finite state-space HMM or when the state-space is compact when the Markov kernel
$\boldsymbol{m_{\param}}$ is sufficiently regular. As noted in \cite{lecorff:fort:2011}, it can fail to hold in quite general situations. Nevertheless, the exponential forgetting
property needed to ensure the convergence results could be checked under weaker assumptions (see \cite{douc:fort:moulines:priouret:2009} for a Doeblin assumption). However, it would imply quite technical supplementary results out of the scope of this paper.
Examples of
observation sequences satisfying
A\ref{assum:obs} include, for example, stationary $\psi$-irreducible and positive recurrent
Markov chains which are geometrically ergodic (see
e.g.~\cite{meyn:tweedie:1993} for Markov chains theory).

\medskip

We need to control the $\rmL_{p}$-mean error on each block between $\bar S_{\tau_{n+1}}^{T_{n}}(\param_{n}, \bfY)$ and its SMC
approximation. This control is discussed in Section~\ref{sec:theory:SMC} below  when the SMC approximation is computed as described in
Section~\ref{subsec:BOEM:description:SMC}.

\subsection{$\rmL_{p}$-error of the SMC approximation}
\label{sec:theory:SMC}
For each block $n$, denote by $\{\adjfunc{t,n}{}{}\}_{t\leq \tau_{n+1}}$ and
$\{\kiss{t,n}{}\}_{t\leq \tau_{n+1}}$ respectively the adjustment multipliers and the
instrumental kernels in the SMC propagation step (see
\eqref{eq:instrumental-distribution-filtering}). For all $y\in\Yset$, define
\[
\ewght{+}{}(y)=\underset{\param\in\paramset}{\sup}\,\underset{\substack{(x,x^{'})\in\Xset\times\Xset \\ t\geq 0, n\geq 0}}{\sup} \dfrac{m_{\param}(x,x^{\prime})g_{\param}(x',y)}{\adjfunc{t,n}{}{x}\kiss{t,n}{}(x,x^{\prime})}  \eqsp.
\]
Consider the following assumptions.
\begin{hypA}
\label{assum:dub:lec}
  $|\adjfunc{}{}{}|_{\infty}\eqdef\sup_{t,n}|\adjfunc{t,n}{}{}|_{\infty} <
  \infty$. 
\end{hypA}
\begin{hypA}
\label{assum:moment:w}
\hspace{-0.15cm}-($p$)
$\lpnorm{\frac{\ewght{+}{}(\bfY_{0})}{b_{-}(\bfY_{0})}}{p}<+\infty\eqsp.$
\end{hypA}
In the case of the {\em Bootstrap filter}, A\ref{assum:dub:lec} holds (since
$v_{t,n} =1$) and  $\ewght{+}{}(y)=
\underset{\param\in\paramset}{\sup}\,\underset{\substack{x\in\Xset}}{\sup}
\,g_{\param}(x,y)$.

\begin{proposition}
\label{prop:check:lp:control}
Let $S : \Xset^{2}\times\Yset\longrightarrow \Rset^{d}$ be a measurable
function s.t. A\ref{assum:moment:sup}-($\bar p$) holds for some
$\bar p >2$. Assume A\ref{assum:strong},
  A\ref{assum:obs}, A\ref{assum:dub:lec}. Define $\Delta p \eqdef 2\bar p p/(\bar p- p)$ and assume A\ref{assum:moment:w}-($\Delta p$) holds for some $p\in(2,\bar p)$.  Then, there exists a constant $C$ s.t. for all $n\ge 0$,
\[
\normL{p}{\widetilde S_{n}(\param_{n}, \bfY) - \bar S_{\tau_{n+1}}^{T_{n}}(\param_{n}, \bfY)}\leq C\left(\frac{1}{N_{n+1}}+\frac{1}{\tau_{n+1}^{1/2}N_{n+1}^{1/2}}\right)\eqsp,
\]
where $\widetilde S_{n}(\param_{n}, \bfY)$ is computed with the algorithm described in Section~\ref{subsec:BOEM:description:SMC}.
\end{proposition}

\subsection{Asymptotic behavior of the Particle Block Online EM algorithms}
\label{sec:PBONEM}
Following \cite{lecorff:fort:2011}, we address the convergence of the P-BOEM
algorithm as the convergence of a perturbed version of the {\em limiting EM}
recursion.  The following result, which is proved in~\cite[Theorem
4.1.]{lecorff:fort:2011}, shows that when $\tau$ is large, the BOEM statistic
$\bar S^{T}_\tau(\param, \bfY)$ is an approximation of a deterministic quantity
$\mapS(\param)$; the {\em limiting EM} algorithm is the iterative procedure defined by
$\param_{n+1} = \limEMmapt(\param_n)$ where
 \begin{equation}
  \label{eq:limem}
\limEMmapt(\param) \eqdef \bar\param(\mapS(\param)) \eqsp, \forall \param\in\paramset\eqsp;
 \end{equation}
 the mapping $\bar \param$ is given by A\ref{assum:exp}.
\begin{theorem}\label{th:LGN}
  Let $S : \Xset^{2}\times\Yset\longrightarrow \Rset^{d}$ be a measurable function s.t. A\ref{assum:moment:sup}-($1$) holds. Assume A\ref{assum:strong}  and
  A\ref{assum:obs}(\ref{assum:obs:erg}).  For any $\param\in\paramset$, there exists a
  $\PPim$-integrable r.v.  $\CE[\bfY]{}{\param}{S(X_{-1},X_{0},\bfY_{0})}$ s.t.
  for any $T>0$,
  \begin{equation}
   \label{th:LGN:ergodic}
  \bar S_{\tau}^{T}(\param, \bfY)
    \underset{\tau\rightarrow
      +\infty}{\longrightarrow} \mapS(\param) \eqdef \CE[]{}{}{\CE[\bfY]{}{\param}{S(X_{-1},X_{0},\bfY_{0})}}\eqsp,\quad
    \ps{\PPim}
\end{equation}
Moreover, $\param \mapsto \mapS(\param)$ is continuous on $\paramset$.
\end{theorem}
The asymptotic behavior of the limiting EM algorithm is addressed in
\cite[Section~$4.2$]{lecorff:fort:2011}: the main ingredient is that the map
$\limEMmapt$ admits a positive and continuous Lyapunov function $\lyap$ w.r.t. the set 
\begin{equation}
\label{eq:staset}
\Staset\eqdef
\{\param\in\paramset; \limEMmapt(\param)=\param\}\eqsp,
\end{equation}
 i.e. \textit{(i)} $\lyap \circ\limEMmapt(\param) \geq \lyap(\param)$ for any $\param \in \paramset$ and,
\textit{(ii)} for any compact subset $\mathcal{K}$ of $\paramset \setminus
\Staset$, $\inf_{\param \in \mathcal{K}} \lyap \circ \limEMmapt(\param) -
\lyap(\param)>0$. This Lyapunov function is equal to $\exp(\ell(\param))$,
where the contrast function $\ell(\param)$ is the (deterministic) limit of
the normalized log-likelihood of the observations when $T\to +\infty$ (see \cite[Theorem $4.9$]{lecorff:fort:2011-supp}).
 
Theorem~\ref{th:Pbonem:conv} establishes the convergence of the P-BOEM algorithm to the set $\Staset$ defined by \eqref{eq:staset}. The proof of Theorem~\ref{th:Pbonem:conv} is an application of \cite[Theorem~$4.4$]{lecorff:fort:2011}. An additional assumption on the number of particles per block is required to check \cite[A$6$]{lecorff:fort:2011} (note indeed that A\ref{assum:particle} below and Proposition~\ref{prop:check:lp:control} imply the condition in \cite{lecorff:fort:2011} about the $\rmL_{p}$-control of the error).

\begin{hypA}\label{assum:particle} 
There exist $c>0$ and $d\ge (a+1)/2a$ (where $a$ is given by A\ref{assum:block-size}) such that, for all $n\ge 1$, $N_{n} = \lfloor c \tau_{n}^{d}\rfloor$.
\end{hypA}

\begin{theorem}
\label{th:Pbonem:conv}
Assume A\ref{assum:exp}-\ref{assum:strong}, A\ref{assum:moment:sup}-($\bar p$), A\ref{assum:obs}-\ref{assum:dub:lec} and A\ref{assum:particle} for some
$\bar p>2$. Define $\Delta p \eqdef 2\bar p p/(\bar p- p)$ and assume A\ref{assum:moment:w}-($\Delta p$) holds for some $p\in(2,\bar p)$.  Assume in addition that $\lyap(\Staset)$ has an empty interior. Then, there exists $w_{\star}$ s.t. $\{\lyap(\param_{n})\}_{n\geq 0}$ converges almost surely to $w_{\star}$ and
  $\{\param_{n}\}_{n\geq 0}$ converges to
  $\{\param\in\Staset;\lyap(\param)=w_{\star}\}$.
\end{theorem}

The assumption on $\lyap(\Staset)$ made in Theorem~\ref{th:Pbonem:conv} is in common use to prove the convergence of EM based procedures or stochastic approximation algorithms. It is used in \cite{wu:1983} to find the limit points of the classical EM algorithm.  See also \cite{delyon:lavielle:moulines:1999} and \cite{fort:moulines:2003} for the stability of the Monte Carlo EM algorithm and of a stochastic approximation of the EM algorithm. If $\lyap$ is sufficiently regular, Sard's theorem states that $\lyap(\Staset)$ has Lebesgue measure $0$ and hence has an empty interior.
 
Under the assumptions of Theorem~\ref{th:Pbonem:conv}, it can be proved that,
along any converging P-BOEM sequence $\{\param_{n}\}_{n\geq 0}$ to
$\param_{\star}$ in $\Staset$, the averaged P-BOEM statistics $\{
\widetilde{\Sigma}_{n}\}_n$ defined by (\ref{eq:PBonem:averaged}) (see
also~(\ref{eq:tildeSigma:bis})) converge to $\mapS(\param_{\star})$, see Proposition~\ref{prop:conv:averaged}.  Since
$\bar\param$ is continuous, the averaged P-BOEM sequence
$\{\widetilde{\param}_{n}\}_{n\geq 0}$ converges to
$\bar\param(\mapS(\param_{\star}))=\limEMmapt(\param_{\star})$. Since
$\param_\star \in \Staset$, $\limEMmapt(\param_{\star})=\param_{\star}$, showing
that the averaged P-BOEM algorithm has the same limit points as the P-BOEM
algorithm. 

\subsection{Rate of convergence of the Particle Block Online EM algorithms}
\label{sec:averaging}
In this section, we consider a converging P-BOEM sequence $\{\param_n \}_{n
  \geq 0}$ with limiting point $\param_\star \in \Staset$. It can be shown, as
in \cite[Proposition $3.1$]{lecorff:fort:2011-supp}, that the convergence of the
sequence $\{\param_n \}_{n \geq 0}$ is equivalent to the convergence of the
sufficient statistics $\{ \widetilde S_{n}(\param_{n},
\bfY)\}_{n\geq 0}$: along any P-BOEM sequence converging to $\param_{\star}$,
this sequence of sufficient statistics converges to $s_{\star} =
\mapS(\param_{\star})$.  
Let $\limEMmap : \Sset \to \Sset$ be the limiting EM map defined on the space
of sufficient statistics by
 \begin{equation}
 \label{eq:mapping:limitingEM}
 \limEMmap(s) \eqdef \mapS(\bar\param(s))
 \eqsp, \qquad \forall s\in\Sset \eqsp.
\end{equation}
To that goal consider the following assumption.
\begin{hypA}\label{assum:stable:fixpoint}
  \begin{enumerate}[(a)]
  \item $\mapS$ and $\bar\param$ are twice continuously differentiable on $\paramset$ and $\Sset$. 
  \item $\mathrm{sp}(\nabla_{s}\limEMmap(s_{\star}))\in (0,1)$ where $\mathrm{sp}$ denotes the spectral radius.
  \end{enumerate}
 \end{hypA}
 
We will use the following notation: for any sequence of random variables
$\{Z_{n}\}_{n\geq 0}$, write $Z_{n} = O_{\rmL_{p}}(1)$ if
$\limsup_{n}\lpnorm{Z_{n}}{p}<\infty$; and $Z_{n} = O_{\mathrm{a.s}}(1)$ if $\sup_{n}|Z_{n}|<+\infty$ $\ps{\PPim}$

\begin{theorem}
\label{th:rate:sto}
Assume A\ref{assum:exp}-\ref{assum:strong}, A\ref{assum:moment:sup}-($\bar p$), A\ref{assum:obs}-\ref{assum:dub:lec} and A\ref{assum:particle}-\ref{assum:stable:fixpoint} for some
$\bar p>2$. Define $\Delta p \eqdef 2\bar p p/(\bar p- p)$ and assume A\ref{assum:moment:w}-($\Delta p$) holds for some $p\in(2,\bar p)$. Then,
\begin{equation}
  \label{eq:rate:nonaverage:sto}
 \left[\param_{n} - \param_{\star}\right]\1_{\lim_{k}\param_{k}
=\param_{\star}}   = O_{\rmL_{p/2}}\left(\frac{1}{\tau_{n}^{1/2}}\right)  O_{\mathrm{a.s}}\left(1\right)\eqsp.
\end{equation}
On the other hand, for the averaged sequence, 
\begin{equation}
  \label{eq:rate:average:sto}
 \left[\widetilde\param_{n} - \param_{\star}\right]\1_{\lim_{k}\param_{k}=\param_{\star}} = O_{\rmL_{p/2}}\left(\frac{1}{T_n^{1/2}}\right)O_{\mathrm{a.s}}\left(1\right) \eqsp.
\end{equation}
\end{theorem}

The proof of Theorem~\ref{th:rate:sto} is obtained by checking the assumptions of \cite[Theorem $5.1$ and Theorem $5.2$]{lecorff:fort:2011}.

Eq.~(\ref{eq:rate:nonaverage:sto}) shows that the error $\param_{n} - \param_{\star}$ has a
$\rmL_{p/2}$-norm decreasing as $\tau_{n}^{-1/2}$. This result is obtained by assuming $N_n \sim \tau_n^{d}$, with $d\ge (a+1)/2a$, which implies that  the SMC error and
the BOEM error are balanced. Unfortunately, such a rate is obtained after a total number
of observations $T_n$; therefore, as discussed in \cite{lecorff:fort:2011}, it
is quite sub-optimal.  Eq~(\ref{eq:rate:average:sto}) shows that the rate of
convergence equal to the square root of the total number of observations up to
block $n$, can be reached by using the averaged P-BOEM algorithm: the $\rmL_{p/2}$-norm of the error $\widetilde
\param_{n}-\param_{\star}$ has a rate of convergence proportional to $T_{n}^{-1/2}$. Here again, note that since $N_{n}$ is chosen as in A\ref{assum:particle} the SMC error and
the BOEM error are balanced.

\section{Proofs}
\label{sec:proofs:sto}
For a function $h$, define $\osc(h) \eqdef \sup_{z,z'} | h(z)-h(z')|$.
\subsection{Proof of Proposition~\ref{prop:check:lp:control}}
For any $t\in\{0,\dots, \tau_{n+1}\}$, define the $\sigma$-algebra
$\mathcal{F}_{n,t}^{N_{n+1}}$ by
\begin{equation}
  \label{eq:Filtrations:Fnt}
  \mathcal{F}_{n,t}^{N_{n+1}}\eqdef \sigma\left\{\param_n,\bfY_{T_n+1:T_n+t+1},\left(\epart{s}{\ell},\ewght{s}{\ell}\right); \ell\in\{1,\dots,N_{n+1}\}; 0\leq s\leq t\right\}\eqsp.
\end{equation}
We use $S_{s}(x,x')$ as a shorthand notation for $S(x,x',\bfY_s)$.
Under A\ref{assum:strong} and A\ref{assum:dub:lec}, Propositions~B.$5$., ~B.$8$. and~B.$9$. in Appendix~B of \cite{lecorff:fort:2011b} can be applied so that
\begin{equation}
\label{eq:diff:sto}
\mathbb{E}_{}\left[\left\vert \widetilde S_{n}(\param_{n}, \bfY) - \bar S_{\tau_{n+1}}^{T_{n}}(\param_{n}, \bfY)\right\vert^{p}\right]\leq C\left(I_{1,n}+I_{2,n}\right)\eqsp,
\end{equation}
where 
\begin{align*}
I_{1,n}&\eqdef \frac{1}{\tau_{n+1}^{\frac{p}{2}+1}N_{n+1}^{\frac{p}{2}}} 
\times\sum_{t=0}^{\tau_{n+1}}\mathbb{E}_{}\left[\left|\frac{\ewght{+}{}(\bfY_{t+T_n})}{b_{-}(\bfY_{t+T_{n}})}\sum_{s=1}^{\tau_{n+1}}\rho^{|t-s|}\osc\{S_{s+T_n}\}\right|^{p}\right]\eqsp,\\
I_{2,n}&\eqdef \frac{1}{\tau_{n+1}N_{n+1}^{p}}\,\\
&\hspace{0.9cm}\times\sum_{t=0}^{\tau_{n+1}}\mathbb{E}_{}\left[\left|\frac{\ewght{+}{}(\bfY_{t+T_n})}{b_{-}(\bfY_{t+T_{n}})}\right|^{2p}\mathbb{E}_{}\left[\left|\sum_{s=1}^{\tau_{n+1}}\rho^{|t-s|}\osc\{S_{s+T_n}\}\right|^{\bar p}\middle|\mathcal{F}_{n,t-1}^{N_{n+1}}\right]^{p/\bar p}\right]\eqsp.
\end{align*}
By the H\"older inequality applied with $\alpha \eqdef \bar p /p\geq 1$ and $\beta^{-1} \eqdef 1 -
\alpha^{-1}$,
\begin{equation*}
I_{1,n}\leq \frac{1}{\tau_{n+1}^{\frac{p}{2}+1}N_{n+1}^{\frac{p}{2}}}\sum_{t=1}^{\tau_{n+1}}\normL{\bar p}{\sum_{s=1}^{\tau_{n+1}}\rho^{\vert t-s\vert}\osc \{S_{s+T_{n}}\}}^{p}\times\normL{2\bar p p /(\bar p -p)}{\frac{\ewght{+}{}(\bfY_{t+T_{n}})}{b_{-}(\bfY_{t+T_{n}})}}^{2p} \eqsp.
\end{equation*}
By A\ref{assum:strong}, A\ref{assum:moment:sup}-($\bar p$), A\ref{assum:obs}\eqref{assum:obs:erg} and A\ref{assum:moment:w}-($\Delta p$), we have
\begin{equation*}
I_{1,n}\leq\frac{C}{\tau_{n+1}^{\frac{p}{2}}N_{n+1}^{\frac{p}{2}}}\eqsp.
\end{equation*}
Using similar arguments for $I_{2,n}$ yields
$I_{2,n}\leq C \, N_{n+1}^{-p}$, which concludes the proof. 
\subsection{ $\rmL_{p}$-controls }
\begin{proposition}
\label{prop:lpcontrol:sto}
Let $S : \Xset^{2}\times\Yset\longrightarrow \Rset^{d}$ be a measurable
function s.t. A\ref{assum:moment:sup}-($\bar p$) holds for some $\bar p >2$. Assume
A\ref{assum:strong}, A\ref{assum:obs}, A\ref{assum:dub:lec} and A\ref{assum:moment:w}-($\Delta p$) for some $p\in(2,\bar
p)$, where $\Delta
p \eqdef 2\bar p p/(\bar p- p)$. There exists a constant $C$
s.t. for any $n\geq 1$,
\begin{equation*}
  \lpnorm{\widetilde S_{n}(\param_{n}, \bfY) - \mapS(\param_n)}{p}
{}\leq  C\left(\frac{1}{\sqrt{\tau_{n+1}}}+\frac{1}{N_{n+1}}\right) \eqsp.
\end{equation*}
\end{proposition}

\begin{proof} 
Under A\ref{assum:strong}, A\ref{assum:moment:sup}-($\bar p$) and A\ref{assum:obs}, by \cite[Theorem~$4.1$]{lecorff:fort:2011}, there exists a constant $C$ s.t.
\[
\lpnorm{\bar S_{\tau_{n+1}}^{T_{n}}(\param_n, \bfY)-\mapS(\param_n)}{p} \leq
\frac{C}{\sqrt{\tau_{n+1}}} \eqsp.
\]
Moreover, under A\ref{assum:dub:lec} and A\ref{assum:moment:w}-($\Delta p$), by Proposition~\ref{prop:check:lp:control}, we have
\begin{equation*}
\normL{p}{\widetilde S_{n}(\param_n, \bfY) - \bar S_{\tau_{n+1}}^{T_n}(\param_n, \bfY)}\leq C\left(\frac{1}{N_{n+1}}+\frac{1}{\tau^{1/2}_{n+1} N_{n+1}^{1/2}}\right)\eqsp,
\end{equation*}
which concludes the proof.
\end{proof}

\begin{proposition}
\label{prop:conv:averaged}
Let $S : \Xset^{2}\times\Yset\longrightarrow \Rset^{d}$ be a measurable
function s.t. A\ref{assum:moment:sup}-($\bar p$) holds for some $ \bar p >2$.
Assume A\ref{assum:strong}, A\ref{assum:obs}-\ref{assum:block-size}, A\ref{assum:dub:lec}-\ref{assum:particle} and A\ref{assum:moment:w}-($\Delta p$) for some $p\in(2,\bar
p)$, where $\Delta
p \eqdef 2\bar p p/(\bar p- p)$. Let $\{\param_n\}_n$ be the P-BOEM sequence. For any
$\param_{\star}\in\paramset$, on the set
$\{\lim_{n}\param_{n}=\param_{\star}\}$,
\[
\widetilde \Sigma_{n} \longrightarrow \mapS (\param_{\star})\eqsp,\quad \ps{\PPim}\eqsp,
\]
where $\mapS$ is defined in \eqref{th:LGN:ergodic} and $\widetilde \Sigma_{n}$ in \eqref{eq:PBonem:averaged}.
\end{proposition}
\begin{proof}
By \eqref{eq:PBonem:averaged}, $\widetilde\Sigma_{n}$ can be written as
\begin{equation}
\label{eq:Sigmatilde:decomp}
\widetilde\Sigma_{n} =
\frac{1}{T_n} \sum_{j=1}^n \tau_j  \, \left[\widetilde
S_{j-1}(\param_{j-1}, \bfY)-\mapS(\param_{j-1})\right] +  \frac{1}{T_n} \sum_{j=1}^n \tau_j  \, \mapS(\param_{j-1})\eqsp.
\end{equation}
By Theorem~\ref{th:LGN}, $\mapS$ is continuous so, by the Cesaro Lemma, the second term in the right-hand side of \eqref{eq:Sigmatilde:decomp} converges to $\mapS(\param_{\star})$ $\PPim$-a.s., on the set $\{\lim_{n}\param_{n}=\param_{\star}\}$. By Proposition~\ref{prop:lpcontrol:sto},
there exists a constant $C$ such that  for any $n$,
\begin{equation*}
  \lpnorm{\widetilde S_{n}(\param_{n}, \bfY) - \mapS(\param_n)}{p}
\leq  C\left(\frac{1}{\sqrt{\tau_{n+1}}}+\frac{1}{N_{n+1}}\right) \eqsp.
\end{equation*}
Hence, by A\ref{assum:block-size}, A\ref{assum:particle} and the Borel-Cantelli Lemma,
\[
\left|\widetilde S_{n}(\param_{n}, \bfY) -
  \mapS(\param_n)\right|\longrightarrow
0\eqsp,\quad \ps{\PPim}
\]
The proof is concluded by applying the Cesaro Lemma.
\end{proof}

\appendix

\section{Detailed SMC algorithm}
\label{app:alg}
In this section, we give a detailed description of the SMC algorithm used to
compute sequentially the quantities $\widetilde
S_{n}(\param_{n}, \bfY)$, $n\geq 0$. This is the algorithm proposed by~ \cite{cappe:2011} and
\cite{delmoral:doucet:singh:2010a}.

At each time step, the weighted samples are produced using sequential
importance sampling and sampling importance resampling steps. In
Algorithm~\ref{alg:FSMC}, the instrumental proposition kernel used to select
and propagate the particles is $\instrpostaux{t}{}$ (see
\eqref{eq:instrumental-distribution-filtering} and
\cite{douc:garivier:moulines:olsson:2010,doucet:defreitas:gordon:2001,liu:2001}
for further details on this SMC step).

It is readily seen from the description below that the observations $\bfY_t$
are processed sequentially.
\begin{algorithm}[h]
\caption{Forward SMC step}
\label{alg:FSMC}
\begin{algorithmic}
\REQUIRE $\param_n$, $\tau_{n+1}$, $N$,
    $\bfY_{T_n+1:T_n+\tau_{n+1}}\eqsp.$
\ENSURE $\widetilde S_{n}(\param_n, \bfY)\eqsp.$
  \STATE Sample $\{\xi_{0}^{\ell}\}_{\ell = 1}^{N}$ i.i.d. with distribution $\chi\eqsp.$
  \STATE Set $\omega_{0}^{\ell} = 1/N$ for all $\ell\in\{1,\dots,N\}\eqsp.$
  \STATE Set $R_{0,\param_n}^{\ell} = 0$ for all $\ell\in\{1,\dots,N\}\eqsp.$
  \FOR {$t=1$ to $\tau_{n+1}$} 
  	\FOR{$\ell=1$ to $N$}
	\STATE Conditionally to $(\param_n, Y_{T_n+1:T_n+t}, \{J_{t-1}^{\ell}, \epart{t-1}{\ell}
      \}_{\ell=1}^{N})$, sample independently $(J_t^{\ell},
      \epart{t}{\ell})\sim \instrpostaux{t}{}(i, \rmd x)\eqsp,$
      where $\instrpostaux{t}{}(i, \rmd x) \propto \ewght{t-1}{i} \adjfunc{t}{}{\epart{t-1}{i}} \kiss{t}{}(\epart{t-1}{i},x)\lambda(\rmd x) \eqsp.$
\STATE Set
\[
\ewght{t}{\ell} = \frac{\m_{\param_n}(\epart{t-1}{J_t^{\ell}},\epart{t}{\ell}) g_{\param_n}(\epart{t}{\ell},\bfY_{T_n+t})}{\adjfunc{t}{}{\epart{t-1}{J_t^{\ell}}} \kiss{t}{}(\epart{t-1}{J_t^{\ell}},\epart{t}{\ell})} \eqsp.
\]
\STATE Set
\[
R_{t,\param_n}^{\ell} = \frac{1}{t}\sum_{j=1}^{N}\ewght{t-1}{j}m_{\param_n}(\epart{t-1}{j},\epart{t}{\ell})\frac{ S(\epart{t-1}{j},\epart{t}{\ell},\bfY_{T_n+t})+ (t-1) R_{t-1,\param_n}^{j}}{\sum_{k=1}^{N}\ewght{t-1}{k}\m_{\param_n}(\epart{t-1}{k},\epart{t}{\ell})}
\eqsp.
\]
\ENDFOR
\ENDFOR
\STATE Set 
\[
\widetilde S_{n}(\param_n, \bfY) = \sum_{\ell=1}^{N}\ewght{\tau_{n+1}}{\ell} R_{\tau_{n+1},\param_n}^{\ell}\eqsp.
\]
\end{algorithmic}
\end{algorithm}

\newpage

\section{$\rmL_{p}$-controls of SMC approximations}
\label{app:extended:version}

In this section, we give further
  details on the $\rmL_{p}$ control on each block (see~~(\ref{eq:diff:sto})):
\[
\mathbb{E}_{}\left[\left\vert \widetilde S_{\tau_{n+1}}^{N,T_{n}}(\param_{n}, \bfY) - \bar S_{\tau_{n+1}}^{T_{n}}(\param_{n}, \bfY)\right\vert^{p}\right]\eqsp,
\]
$\bar S_{\tau}^{T}$ is defined by \eqref{eq:rewrite:barS} (we recall that,
$\chi$ being fixed, it is dropped from the notations) and $\widetilde
S_\tau^{N,T}$ is the SMC approximation of $\bar S_{\tau}^{T}$ based on $N$
particles computed as described in Section~\ref{subsec:BOEM:description:SMC}.

The following results are technical lemmas taken from \cite{douc:garivier:moulines:olsson:2010} (stated here for a better clarity) or extensions of the $\rmL_{p}$ controls derived in \cite{dubarry:lecorff:2011}.

\bigskip

Hereafter, ``time $t$'' corresponds to time $t$ in the block
  $n$. Therefore, even if it is not explicit in the notations (in order to make
  them simpler), the following quantities depend upon the observations
  $\bfY_{T_n+1:T_n+\tau_{n+1}}$.

\medskip 

Denote by $\filt{s}^{\param}$ the filtering distribution at
  time $s$, and let
\[
\BK{\filt{t}^{\param}}^{\param}(x, \rmd x') \eqdef \frac{m_\param(x',x)}{\int
  m_\param(u,x) \filt{t}^\param(\rmd u)}\filt{t}^\param(\rmd x')
\]
be the backward kernel smoothing kernel at time $t+1$. For all $0 \leq s \leq
\tau-1$ and for all bounded measurable function $h$ on $\Xset^{\tau-s+1}$,
define recursively $\post{s:\tau}{\tau}^\param[h]$ backward in time, according
to
\begin{equation} \label{eq:smoothing:backw_decomposition_recursion}
    \post{s:\tau}{\tau}^{\param}[h] = \idotsint \BK{\filt{s}^{\param}}^{\param}(x_{s+1}, \rmd x_s) \, \post{s+1:\tau}{\tau}^{\param}(\rmd \chunk{x}{s+1}{\tau}) \, h(\chunk{x}{s}{\tau}) \eqsp,
\end{equation}
starting from $\post{\tau:\tau}{\tau}^{\param}
  =\filt{\tau}^{\param}$. By convention, $\filt{0}^\param = \chi$.

\medskip

For $t \geq 1$, let $\left\{ (\epart{t}{\ell},
    \ewght{t}{\ell})\right\}_{\ell = 1}^{N}$ be the weighted samples obtained as
  described in Section~\ref{subsec:BOEM:description:SMC} (see also
  Algorithm~\ref{alg:FSMC} in Appendix~\ref{app:alg}); it approximates the
  filtering distribution $\filt{t}^\param$. Denote by $\filt{t}^{N,\param}$
  this approximation. For $0 \leq s
\leq \tau-1$, an approximation of the backward kernel can be obtained
$$
\BK{\filt{s}^{N,\param}}(x, h) = \sum_{i = 1}^{N} \frac{\ewght{s}{i}
  \m_{\param}(\epart{s}{i}, x)}{\sum_{\ell = 1}^{N} \ewght{s}{\ell}
  \m_{\param}(\epart{s}{\ell},x)} h\left( \epart{s}{i} \right) \eqsp;
$$
and inserting this expression into \eqref{eq:smoothing:backw_decomposition_recursion} gives the following particle approximation of the fixed-interval smoothing distribution $\post{0:\tau}{\tau}^{\param}[h]$
\begin{equation} \label{eq:forward-filtering-backward-smoothing}
    \post{0:\tau}{\tau}^{N,\param}[h] = \sum_{i_0 = 1}^{N} \dots \sum_{i_{\tau} = 1}^{N}
    \left(\prod_{u=1}^{\tau} \frac{\ewght{u-1}{i_{u-1}} \m_{\param}(\epart{u-1}{i_{u-1}},\epart{u}{i_u})}{\sum_{\ell=1}^{N} \ewght{u-1}{\ell} \m_{\param}(\epart{u-1}{\ell},\epart{u}{i_u})}\right)  \times \frac{\ewght{\tau}{i_{\tau}}}{\sumwght[n]{\tau}} h\left(\epart{0}{i_0}, \dots, \epart{\tau}{i_{\tau}}\right) \eqsp,
\end{equation}
with $\sumwght[N]{\tau}\eqdef\sum_{\ell=1}^{N}\ewght{\tau}{\ell}$. 
\begin{lemma}
\label{lem:DeltaStoDeltaPhi}
Let $\left\{ (\epart{t}{\ell}, \ewght{t}{\ell}), 1 \leq \ell \leq N, 0 \leq t
  \leq \tau_{n+1} \right\}$ be the weighted samples obtained by
  Algorithm~\ref{alg:FSMC} in Appendix~\ref{app:alg}, with
  $\param_n, \tau_{n+1}$, $N$, $\bfY_{T_n+1:T_n+\tau_{n+1}}$. Then,
\begin{multline}
\label{eq:rewrite:smerr}
\left[\widetilde S_{\tau_{n+1}}^{N,T_n}(\param_n, \bfY)-\bar S_{\tau_{n+1}}^{T_n}(\param_n, \bfY)\right] \\= \frac{1}{\tau_{n+1}} 
\left(\phi^{N,\param_n}_{0:\tau_{n+1}|\tau_{n+1}}\left[\addfunc{\tau_{n+1}}\right]-\phi_{0:\tau_{n+1}|\tau_{n+1}}^{\param_n}\left[\addfunc{\tau_{n+1}}\right]  \right) \eqsp,
\end{multline}
where
\begin{equation}
  \addfunc{\tau}(x_{0:\tau}) \eqdef  \sum_{s=1}^{\tau}S(x_{s-1},x_{s},\bfY_{s+T})\eqsp.  \label{eq:addfuncbis} 
\end{equation}
\end{lemma}

\bigskip

For all $t\in\{0,\dots,\tau\}$ and all bounded measurable function $h$ on
$\Xset^{\tau+1}$, define the kernel $\rmL_{t,\tau}:\Xset^{t+1}\times
\sigmaX^{\otimes \tau+1}\rightarrow [0,1]$ by
\begin{equation}\label{eq:ldroit}
\rmL_{t,\tau}^{\param}h(x_{0:t})\eqdef \int \prod_{u=t+1}^{\tau}\m_{\param}(x_{u-1},x_u)g_{\param}(x_u,\bfY_{u+T})h(x_{0:\tau})\lambda(\rmd x_{t+1:\tau})\eqsp;
\end{equation}
by convention, $\rmL_{\tau,\tau}^{\param}h =h$. Let
$\mathcal{L}_{t,\tau}^{N,\param}$ and $\mathcal{L}_{t,\tau}^{\param}$ be two
kernels on $\Xset \times \sigmaX^{\otimes (\tau+1)}$ defined for all
$x_t\in\Xset$ by
\begin{align}
  \mathcal{L}_{t,\tau}^{\param}h(x_t) &\eqdef
  \int\rmB_{\phi_{t-1}^{\param}}^{\param}(x_t,\rmd
  x_{t-1})\cdots\rmB_{\phi_{0}^{\param}}^{\param}(x_1,\rmd
  x_{0})\rmL_{t,\tau}^{\param}h(x_{0:t}) \label{eq:defL}
  \\\mathcal{L}_{t,\tau}^{N,\param}h(x_t) &\eqdef
  \int\rmB_{\phi_{t-1}^{N,\param}}^{\param}(x_t,\rmd
  x_{t-1})\cdots\rmB_{\phi_{0}^{N,\param}}^{\param}(x_1,\rmd
  x_{0})\rmL_{t,\tau}^{\param}h(x_{0:t})\label{eq:defLN}\eqsp.
\end{align}
Note that
\begin{equation}\label{eq:lcal}
\mathcal{L}_{t,\tau}^{\param}\1(x_t) =\int m_\param(x_t,x') g_\param(x',\bfY_{T+t+1}) \ 
    \mathcal{L}_{t+1,\tau}^{\param} \1(x') \lambda(\rmd x') \eqsp.
\end{equation}

Lemma~\ref{lem:rew:err}, Proposition~\ref{prop:DeltaStoC-D}, Lemma~\ref{lem:esperancecond} and~\ref{Lem:Upperbounds} can be found in \cite{douc:garivier:moulines:olsson:2010}.
\begin{lemma}
\label{lem:rew:err}
  Let $\left\{ (\epart{t}{\ell}, \ewght{t}{\ell}), 1 \leq \ell \leq N, 0 \leq t
    \leq \tau_{n+1} \right\}$ be the weighted samples obtained by
    Algorithm~\ref{alg:FSMC} in Appendix~\ref{app:alg}, with 
    $\param_n, \tau_{n+1}$, $N$, $\bfY_{T_n+1:T_n+\tau_{n+1}}$. Then,
\begin{equation}\label{Eq:Err}
    \phi^{N,\param_n}_{0:\tau_{n+1}|\tau_{n+1}}\left[h\right]-\phi_{0:\tau_{n+1}|\tau_{n+1}}^{\param_n}\left[h\right]
    = \sum_{t=0}^{\tau_{n+1}}\frac{\sum_{\ell=1}^{N}\ewght{t}{\ell} \  
      G_{t,\tau_{n+1}}^{N,\param_n}h(\epart{t}{\ell})}{\sum_{\ell=1}^{N}\ewght{t}{\ell} \ 
      \mathcal{L}_{t,\tau_{n+1}}^{\param_n}\1(\epart{t}{\ell})}\eqsp,
\end{equation}
with  $G_{t,\tau}^{N,\param}$ is a kernel on $\Xset \times \sigmaX^{\otimes (\tau+1)}$ defined, for all $x\in\Xset$ and all bounded and measurable function $h$ on $\Xset^{\tau+1}$, by
\begin{equation}
\label{eq:defG}
G_{t,\tau}^{N,\param}h(x)\eqdef \mathcal{L}_{t,\tau}^{N,\param}h(x) - \frac{\phi_{t-1}^{N,\param}[\mathcal{L}_{t-1,\tau}^{N,\param}h]}{\phi_{t-1}^{N,\param}[\mathcal{L}_{t-1,\tau}^{N,\param}\1]}\mathcal{L}_{t,\tau}^{N,\param}\1(x)\eqsp.
\end{equation}
\end{lemma}
\begin{proof}
  By definition of $\rmL_{t,\tau}^{\param}$,
\begin{equation*}
\phi_{0:\tau|\tau}^{\param}[h] = \frac{\phi_{0:t|t}^{\param}\left[\rmL_{t,\tau}^{\param}h\right]}{\phi_{0:t|t}^{\param}\left[\rmL_{t,\tau}^{\param}\1\right]}\eqsp.
\end{equation*}
We write
\[
\phi^{N,\param}_{0:\tau|\tau}\left[h\right]-\phi_{0:\tau|\tau}^{\param}\left[h\right]
= \sum_{t=0}^{\tau} \left\{
  \frac{\phi_{0:t|t}^{N,\param}\left[\rmL_{t,\tau}^{\param}h\right]}{\phi_{0:t|t}^{N,\param}\left[\rmL_{t,\tau}^{\param}\1\right]}-\frac{\phi_{0:t-1|t-1}^{N,\param}\left[\rmL_{t-1,\tau}^{\param}h\right]}{\phi_{0:t-1|t-1}^{N,\param}\left[\rmL_{t-1,\tau}^{\param}\1\right]}
\right \} \nonumber\eqsp,
\]
where we used the convention
\begin{equation*}
\frac{\phi_{0:-1|-1}^{N,\param}\left[\rmL_{-1,\tau}^{\param}h\right]}{\phi_{0:-1|-1}^{N,\param}\left[\rmL_{-1,\tau}^{\param}\1\right]} = \frac{\chi\left[\rmL_{0,\tau}^{\tau}h\right]}{\chi\left[\rmL_{0,\tau}^{\param}\1\right]} = \phi_{0:\tau|\tau}^{\param}[h]\eqsp.
\end{equation*}
We have for all $0\leq t\leq \tau$,
 \[   \phi_{0:t|t}^{N,\param}\left[\rmL_{t,\tau}h\right]
   =\int\phi_{t}^{N,\param}(\rmd x_{t}) \,  \prod_{j=0}^{t-1}  \rmB_{\phi_{j}^{N,\param}}(x_{j+1},\rmd
   x_{j})  \  \rmL_{t,\tau}^{\param}h(x_{0:t}) = \phi_{t}^{N,\param}[\mathcal{L}_{t,\tau}^{N,\param}h]\eqsp.
 \]
Therefore, for all $1\leq t \leq \tau$,
\[\frac{\phi_{0:t|t}^{N,\param}[\rmL_{t,\tau}^{\param}h]}{\phi_{0:t|t}^{N,\param}[\rmL_{t,\tau}^{\param}\1]}-\frac{\phi_{0:t-1|t-1}^{N,\param}[\rmL_{t-1,\tau}^{\param}h]}{\phi_{0:t-1|t-1}^{N,\param}[\rmL_{t-1,\tau}^{\param}\1]}
  =\frac{\phi_{t}^{N,\param}[\mathcal{L}_{t,\tau}^{N,\param}h]}{\phi_{t}^{N,\param}[\mathcal{L}_{t,\tau}^{N,\param}\1]}
  -
  \frac{\phi_{t-1}^{N,\param}[\mathcal{L}_{t-1,\tau}^{N,\param}h]}{\phi_{t-1}^{N,\param}[\mathcal{L}_{t-1,\tau}^{N,\param}\1]}
  =\frac{\phi_{t}^{N,\param}[G_{t,\tau}^{N,\param}h
    ]}{\phi_{t}^{N,\param}[\mathcal{L}_{t,\tau}^{N,\param}\1]}\ \eqsp.
\]
\end{proof}
 
\begin{proposition}
\label{prop:DeltaStoC-D}
 Let $\left\{ (\epart{t}{\ell}, \ewght{t}{\ell}), 1 \leq \ell \leq N, 0 \leq t
    \leq \tau_{n+1} \right\}$ be  the weighted samples obtained by Algorithm~\ref{alg:FSMC} in
  Appendix~\ref{app:alg}, with input variables $\param_n, \tau_{n+1}$, $N$,
  $\bfY_{T_n+1:T_n+\tau_{n+1}}$.  Then, 
\begin{multline*}
\left[\widetilde S_{\tau_{n+1}}^{N,T_n}(\param_n, \bfY)-\bar S_{\tau_{n+1}}^{T_n}(\param_n, \bfY)\right]  =  \frac{1}{\tau_{n+1}}\sum_{t=0}^{\tau_{n+1}}D_{t,\tau_{n+1}}^{N,\param_n}(\addfunc{\tau_{n+1}})\\
+  \frac{1}{\tau_{n+1}}\sum_{t=0}^{\tau_{n+1}}C_{t,\tau_{n+1}}^{N,\param_n}(\addfunc{\tau_{n+1}})\eqsp,
\end{multline*}
where  $\addfunc{\tau}$ is given by (\ref{eq:addfuncbis}) and
\begin{align}
\label{eq:defD}
& D_{t,\tau_{n+1}}^{N,\param_n}(h) \eqdef
\frac{\phi_{t-1}^{N,\param_n}[\adjfunc{t}{}{}]}{\phi_{t-1}^{N,\param_n}\left[\frac{\mathcal{L}_{t-1,\tau_{n+1}}^{\param_n}\1}{|\mathcal{L}_{t,\tau_{n+1}}^{\param_n}\1|_{\infty}}\right]}N^{-1}\sum_{\ell=1}^{N}\ewght{t}{\ell}
\frac{G_{t,\tau_{n+1}}^{N,\param_n}h(\epart{t}{\ell})}{|\mathcal{L}_{t,\tau_{n+1}}^{\param_n}\1|_{\infty}}\eqsp;  \\
& C_{t,\tau_{n+1}}^{N}(h) \eqdef \left[\frac{1}{N^{-1}\sum_{i=1}^{N}\ewght{t}{i} \frac{\mathcal{L}_{t,\tau_{n+1}}^{\param_n}\1(\epart{t}{i})}{|\mathcal{L}_{t,\tau_{n+1}}^{\param_n}\1|_{\infty}}}-\frac{\phi_{t-1}^{N,\param_n}[\adjfunc{t}{}{}]}{\phi_{t-1}^{N,\param_n}\left[\frac{\mathcal{L}_{t-1,\tau_{n+1}}^{\param_n}\1}{|\mathcal{L}_{t,\tau_{n+1}}^{\param_n}\1|_{\infty}}\right]}\right] \label{eq:defC}\\
&\hspace{7cm}\times N^{-1}\sum_{\ell=1}^{N}\ewght{t}{\ell} \frac{G_{t,\tau_{n+1}}^{N,\param_n}h(\epart{t}{\ell})}{|\mathcal{L}_{t,\tau_{n+1}}^{\param_n}\1|_{\infty}}
\eqsp.\nonumber
\end{align}
\end{proposition}
\begin{proof}
  \eqref{Eq:Err} can be rewritten as follows:
\begin{equation}
\label{eq:rewrite:smerrCD}
\phi^{N,\param_n}_{0:\tau_{n+1}|\tau_{n+1}}\left[h\right]-\phi_{0:\tau_{n+1}|\tau_{n+1}}^{\param_n}\left[h\right]  = \sum_{t=0}^{\tau_{n+1}}D_{t,\tau_{n+1}}^{N,\param_n}(h) + \sum_{t=0}^{\tau_{n+1}}C_{t,\tau_{n+1}}^{N,\param_n}(h)\eqsp.
\end{equation}
The proof is concluded by Lemma~\ref{lem:DeltaStoDeltaPhi}.
\end{proof}
For any $t\in\{0,\dots, \tau_{n+1}\}$, we recall the definition of $\mathcal{F}_{n,t}^{N}$ given by \eqref{eq:Filtrations:Fnt}
\begin{equation*}
  \mathcal{F}_{n,t}^{N}= \sigma\left\{\param_n,\bfY_{T_n+1:T_n+t+1},\left(\epart{s}{\ell},\ewght{s}{\ell}\right); \ell\in\{1,\dots,N\}; 0\leq s\leq t\right\}\eqsp,
\end{equation*}
where $\left\{ (\epart{t}{\ell}, \ewght{t}{\ell})\right\}_{\ell = 1}^{N}$ are
the weighted samples obtained by Algorithm~\ref{alg:FSMC} in
Appendix~\ref{app:alg}, with input variables $\param_n, \tau_{n+1}$, $N$,
$\bfY_{T_n+1:T_n+\tau_{n+1}}$. 

\begin{lemma}
\label{lem:esperancecond}
Let $\left\{ (\epart{t}{\ell}, \ewght{t}{\ell}), 1 \leq \ell \leq N, 0 \leq t
  \leq \tau_{n+1} \right\}$ be the weighted samples obtained by
  Algorithm~\ref{alg:FSMC} in Appendix~\ref{app:alg}, with 
  $\param_n, \tau_{n+1}$, $N$, $\bfY_{T_n+1:T_n+\tau_{n+1}}$. Then, for any $1
  \leq t \leq \tau_{n+1}$ and any $1 \leq \ell \leq N$,
\begin{equation}
\label{eq:whExpect}
\CExpparam{\ewght{t}{\ell}h(\epart{t}{\ell})}{\mathcal{F}_{n,t-1}^{N}}{} =\frac{\phi_{t-1}^{N,\param_n}\left[ \int \m_{\param_n}(\cdot,x) g_{\param_n}(x, \bfY_{T_n+t}) \ h(x) \ \lambda(\rmd x)\right]}{\phi_{t-1}^{N,\param_n}[\adjfunc{t}{}{}]}\eqsp.
\end{equation}
\end{lemma}
\begin{proof}
  By definition of the weighted particles,
 \begin{align*}
&\CExpparam{\ewght{t}{\ell}h(\epart{t}{\ell})}{\mathcal{F}_{n,t-1}^{N}}{} \nonumber\\
&= \CExpparam{\frac{\m_{\param_n}(\epart{t-1}{I_{t}^{1}},\epart{t}{1}) g_{\param_n}(\epart{t}{1},\bfY_{t+T_n})}{\adjfunc{t}{}{\epart{t-1}{I_{t}^{1}}} \kiss{t}{}(\epart{t-1}{I_{t}^{1}},\epart{t}{1})}h(\epart{t}{1})}{\mathcal{F}_{n,t-1}^{N}}{} \nonumber\\
&= \left(\sum_{i=1}^{N}\ewght{t-1}{i} \adjfunc{t}{}{}{(\epart{t-1}{i})}\right)^{-1}\sum_{i=1}^{N}\int\ewght{t-1}{i} \adjfunc{t}{}{\epart{t-1}{i}}\kiss{t}{}(\epart{t-1}{i},x)\\
&\hspace{5.5cm}\times\frac{\m_{\param_n}(\epart{t-1}{i},x) g_{\param_n}(x,\bfY_{t+T_n})}{\adjfunc{t}{}{\epart{t-1}{i}} \kiss{t}{}{}(\epart{t-1}{i},x)}h(x)\lambda(\rmd x) \nonumber\\
&= \left(\sum_{i=1}^{N}\ewght{t-1}{i} \adjfunc{t}{}{}{(\epart{t-1}{i})}\right)^{-1}\sum_{i=1}^{N}\int\ewght{t-1}{i}\m_{\param_n}(\epart{t-1}{i},x) g_{\param_n}(x,\bfY_{t+T_n})h(x)\lambda(\rmd x) \eqsp.
\end{align*}
\end{proof}

\begin{lemma}
\label{Lem:Upperbounds}
Assume A\ref{assum:strong} and A\ref{assum:dub:lec}.  Let $\left\{
  (\epart{t}{\ell}, \ewght{t}{\ell}), 1 \leq \ell \leq N, 0 \leq t \leq
  \tau_{n+1} \right\}$ be the weighted samples obtained by Algorithm~\ref{alg:FSMC}
in Appendix~\ref{app:alg}, with input variables $\param_n, \tau_{n+1}$, $N$,
$\bfY_{T_n+1:T_n+\tau_{n+1}}$.
\begin{enumerate}[(i)]
\item \label{Lem:Upperbounds:ind} For any $t\in\{0,\dots, \tau_{n+1}\}$ and any
measurable function $h$ on $\Xset^{\tau_{n+1}+1}$, the random variables $\displaystyle \left\{ \ewght{t}{\ell} \  G^{N,\param_n}_{t,\tau_{n+1}}h(\epart{t}{\ell}) \ |\mathcal{L}_{t,\tau_{n+1}}^{\param_n}\1|_\infty^{-1}\right\}_{\ell=1}^{N}$ are:
    \begin{enumerate}
    \item \label{Lem:Upperbounds:ind:claim1} conditionally independent and identically distributed given $\mathcal{F}_{n,t-1}^{N}$\eqsp,
    \item \label{Lem:Upperbounds:ind:claim2}centered conditionally to $\mathcal{F}_{n,t-1}^{N}$\eqsp.
    \end{enumerate}
  \item \label{Lem:Upperbounds:boundup} For any $t\in\{0,\dots, \tau_{n+1}\}$:
\begin{equation}\label{eq:bound-G}
\left| \dfrac{G^{N,\param_n}_{t,\tau_{n+1}}\mathsf{S}_{\tau_{n+1}}(\epart{t}{\ell})}{|\mathcal{L}_{t,\tau_{n+1}}^{\param_n}\1|_\infty}\right|\leq \sum_{s=1}^{\tau_{n+1}}\rho^{|t-s|}\osc\{S(\cdot,\cdot,\bfY_{s+T_n})\}\eqsp,
\end{equation}
where $\mathsf{S}_{\tau}$ is defined by \eqref{eq:addfuncbis}.
\item \label{Lem:Upperbounds:boundlow}
For all $x\in\Xset$ and any $t\in\{0,\dots, \tau_{n+1}\}$, 
\[
\displaystyle
\dfrac{\mathcal{L}_{t,\tau_{n+1}}^{\param_n}\1(x)}{|\mathcal{L}_{t,\tau_{n+1}}^{\param_n}\1|_\infty}
\geq \dfrac{\sigma_-}{\sigma_+} \eqsp, \qquad \displaystyle
\dfrac{\mathcal{L}_{t-1,\tau_{n+1}}^{\param_n}\1(x)}{|\mathcal{L}_{t,\tau_{n+1}}^{\param_n}\1|_\infty}
\geq \dfrac{\sigma_-^{2}}{\sigma_+}b_{-}(\bfY_{t+T_n}) \eqsp.
\]
\end{enumerate}
\end{lemma}

\begin{proof}
  The proof of \eqref{Lem:Upperbounds:ind} is given by \GisCentered.

{\em Proof of \eqref{Lem:Upperbounds:boundup}}. Let $\Pi_{s-1:s,\tau}$ be the operator which associates to any bounded and measurable function $h$ on $\Xset\times\Xset$ the function $\Pi_{s-1:s,\tau}h$ given, for any $(x_0,\dots,x_{\tau}) \in \Xset^{\tau+1}$, by
\begin{equation*}
\Pi_{s-1:s,\tau}h(x_{0:\tau}) \eqdef h(x_{s-1:s})\eqsp.
\end{equation*}
Using this notation, we may write $\addfunc{\tau} = \sum_{s=1}^{\tau}
\Pi_{s-1:s,\tau}S(\cdot,\cdot,\bfY_{s+T})$ and $G^{N,\param}_{t,\tau}\addfunc{\tau} =
\sum_{s=1}^{\tau}G^{N,\param}_{t,\tau}\Pi_{s-1:s,\tau}S(\cdot,\cdot,\bfY_{s+T})$.  Following the
same lines as in \GNorm,
\begin{align*}
|G^{N,\param}_{t,\tau}\Pi_{s-1:s,\tau}S(\cdot,\cdot,\bfY_{s+T})|_\infty &\leq \rho^{s-1-t} \osc(S(\cdot,\cdot,\bfY_{s+T})) |\mathcal{L}_{t,\tau}^{\param}\1|_\infty \quad \mbox{if}\quad t\leq s-1\eqsp,\\
|G^{N,\param}_{t,\tau}\Pi_{s-1:s,\tau}S(\cdot,\cdot,\bfY_{s+T})|_\infty &\leq \rho^{t-s} \osc(S(\cdot,\cdot,\bfY_{s+T})) |\mathcal{L}_{t,\tau}^{\param}\1|_\infty \qquad \mbox{if}\quad t\geq s\eqsp.
\end{align*}
Consequently,
\begin{multline*}
  \left|G^{N,\param}_{t,\tau}\addfunc{\tau}\right|_\infty \leq
  \sum_{s=1}^{\tau} |G_{t,\tau}^{N,\param}\Pi_{s-1:s,\tau}S(\cdot,\cdot,\bfY_{s+T})|_\infty\\ \leq
   \left(\sum_{s=1}^{\tau}\rho^{|t-s|}\osc\{S(\cdot,\cdot,\bfY_{s+T})\}\right)
  |\mathcal{L}_{t,\tau}^{\param}\1|_\infty\eqsp,
\end{multline*}
which shows \eqref{Lem:Upperbounds:boundup}. 

{\em Proof of \eqref{Lem:Upperbounds:boundlow}}.
By the definition \eqref{eq:defL}, for all $x\in\Xset$ and all $t\in\{1,\dots,\tau\}$,
\begin{multline*}
\mathcal{L}_{t,\tau}^{\param}\1(x) = \int m_{\param}(x,x_{t+1})g_{\param}(x_{t+1},\bfY_{t+T+1}) \ \\
\times\prod_{u=t+2}^{\tau}\m_{\param}(x_{u-1},\rmd x_u)g_{\param}(x_{u},\bfY_{u+T})\lambda(\rmd x_{t+1:\tau})\eqsp.
\end{multline*}
Hence, by A\ref{assum:strong},
\begin{align*}
\left|\mathcal{L}_{t,\tau}^{\param}\1\right|_{\infty}&\leq \sigma_+\int g_{\param}(x_{t+1},\bfY_{t+T+1})\mathcal{L}_{t+1,\tau}^{\param}\1(x_{t+1})\lambda(\rmd x_{t+1})\\
\mathcal{L}_{t,\tau}^{\param}\1(x)&\geq \sigma_- \int g_{\param}(x_{t+1},\bfY_{t+T+1})\mathcal{L}_{t+1,\tau}^{\param}\1(x_{t+1})\lambda(\rmd x_{t+1})\eqsp,
\end{align*}
which concludes the proof of the first statement. By (\ref{eq:lcal}),
A\ref{assum:strong} and \eqref{eq:b+:b-},
\begin{equation*}
\dfrac{\mathcal{L}_{t-1,\tau}^{\param}\1(x)}{|\mathcal{L}_{t,\tau}^{\param}\1|_\infty} =\int \m_{\param}(x,x^{\prime})g_{\param}(x^{\prime},\bfY_{t+T}) \dfrac{\mathcal{L}_{t,\tau}^{\param}\1(x^{\prime})}{|\mathcal{L}_{t,\tau}^{\param}\1|_\infty}\lambda(\rmd x')
 \geq \dfrac{\sigma_-^{2}}{\sigma_+}b_{-}(\bfY_{t+T})\eqsp.
\end{equation*}
\end{proof}

The proofs of Propositions~\ref{Prop:NormD} and \ref{Prop:NormC} follow the same lines as \cite[Propositions $1$-$2$]{dubarry:lecorff:2011}. The upper bounds given here provide an explicit dependence on the observations.

\begin{proposition}\label{Prop:NormD}
  Assume A\ref{assum:strong} and A\ref{assum:dub:lec}. Let $\left\{
  (\epart{t}{\ell}, \ewght{t}{\ell}), 1 \leq \ell \leq N, 0 \leq t \leq
  \tau_{n+1} \right\}$ be the weighted
  samples obtained by Algorithm~\ref{alg:FSMC} in Appendix~\ref{app:alg}, with
  input variables $\param_n, \tau_{n+1}$, $N$, $\bfY_{T_n+1:T_n+\tau_{n+1}}$.
  For all $p > 1$, there exists a constant $C$ such that 
\begin{multline}\label{Eq:NormD}
\mathbb{E}_{}\left[\left|\sum_{t=0}^{\tau_{n+1}}D_{t,\tau_{n+1}}^{N,\param_n}(\addfunc{\tau_{n+1}})\right|^{p}\right]\\ 
\leq C \frac{\tau_{n+1}^{(\frac{p}{2}-1) \vee 0}}{N^{p-(\frac{p}{2}\vee 1)}} \sum_{t=0}^{\tau_{n+1}}\mathbb{E}_{}\left[\left|\frac{\ewght{+}{}(\bfY_{t+T_n})}{b_{-}(\bfY_{t+T_{n}})}\sum_{s=1}^{\tau_{n+1}}\rho^{|t-s|}\osc\{S(\cdot,\cdot, \bfY_{s+T_n})\}\right|^{p}\right]\eqsp.
\end{multline}
where $D_{t,\tau}^{N,\param}$ is defined in \eqref{eq:defD}.
\end{proposition}

\begin{proof}  By Lemma \ref{Lem:Upperbounds}\eqref{Lem:Upperbounds:boundlow},
\begin{equation*}
\frac{\phi_{t-1}^{N,\param_n}[\adjfunc{t}{}{}]}{\phi_{t-1}^{N,\param_n}\left[\frac{\mathcal{L}_{t-1,\tau_{n+1}}^{\param_n}\1}{|\mathcal{L}_{t,\tau_{n+1}}^{\param_n}\1|_{\infty}}\right]} \leq  \frac{\sigma_{+} |\adjfunc{}{}{}|_{\infty}}{\sigma_{-}^{2}b_{-}(\bfY_{t+T_n})}\eqsp.
\end{equation*} 
By Lemma~\ref{Lem:Upperbounds}(\ref{Lem:Upperbounds:ind}) and since $\param_{n}$ is $\mathcal{F}_{n,t}^{N}$-measurable for all $t\in\{0,\dots,\tau_{n+1}\}$, $\left\{D_{t,\tau_{n+1}}^{N,\param_n}(\addfunc{\tau_{n+1}}),\mathcal{F}_{n,t}^{N}\right\}_{0\leq
  t\leq \tau_{n+1}}$ is a martingale difference. Since $p> 1$,
Burkholder's inequality (see \cite[Theorem 2.10, page 23]{hall:heyde:1980})
states the existence of a constant $C$ depending only on $p$ such that:
\begin{equation*}
\mathbb{E}_{}\left[\left|\sum_{t=0}^{\tau_{n+1}}D_{t,\tau_{n+1}}^{N,\param_n}(\addfunc{\tau_{n+1}})\right|^{p}\right]\leq C\mathbb{E}\left[\left|\sum_{t=0}^{\tau_{n+1}} \left|D_{t,\tau_{n+1}}^{N,\param_n}(\addfunc{\tau_{n+1}}) \right|^2\right|^{p/2}\right]\eqsp.
\end{equation*}
Hence,
\begin{multline*}
\mathbb{E}_{}\left[\left|\sum_{t=0}^{\tau_{n+1}}D_{t,\tau_{n+1}}^{N,\param_n}(\addfunc{\tau_{n+1}})\right|^{p}\right] \leq C   \left(\frac{\sigma_{+} |\adjfunc{}{}{}|_{\infty}}{\sigma_{-}^{2}}\right)^p\\
  \times\mathbb{E}_{}\left[\left|\sum_{t=0}^{\tau_{n+1}}\left|N^{-1}\sum_{\ell=1}^{N}\frac{\ewght{t}{\ell}}{b_{-}(\bfY_{t+T_n})}
      \frac{G_{t,\tau_{n+1}}^{N,\param_n}\addfunc{\tau_{n+1}}(\epart{t}{\ell})}{|\mathcal{L}_{t,\tau_{n+1}}^{\param_n}\1|_{\infty}}\right|^2\right|^{p/2}\right]\eqsp,
\end{multline*}
which implies, using the convexity inequality $(\sum_{k=1}^\tau a_k)^{p/2} \leq
\tau^{(p/2-1) \vee 0} \sum_{k=1}^\tau a_k^{p/2}$,
\begin{multline*}
\mathbb{E}_{}\left[\left|\sum_{t=0}^{\tau_{n+1}}D_{t,\tau_{n+1}}^{N,\param_n}(\addfunc{\tau_{n+1}})\right|^{p}\right]\leq C\left(\frac{\sigma_{+} |\adjfunc{}{}{}|_{\infty}}{\sigma_{-}^{2}}\right)^p\\
\times \frac{\left(\tau_{n+1}+1\right)^{(\frac{p}{2}-1) \vee 0}}{N^{p}} \sum_{t=0}^{\tau_{n+1}}  \mathbb{E}_{}\left[\left|\frac{1}{b_{-}(\bfY_{t+T_n})} \sum_{\ell=1}^{N} \ewght{t}{\ell} \frac{G_{t,\tau_{n+1}}^{N,\param_n}\addfunc{\tau_{n+1}}(\epart{t}{\ell})}{|\mathcal{L}_{t,\tau_{n+1}}^{\param_n}\1|_{\infty}}\right|^{p}\right]\eqsp.
\end{multline*}
Since $\bfY_{t+T_n}$ and $\param_{n}$ are $\mathcal{F}_{n,t-1}^{N}$-measurable,
\begin{multline*}
\mathbb{E}_{}\left[\left|\frac{1}{b_{-}(\bfY_{t+T_n})} \sum_{\ell=1}^{N} \ewght{t}{\ell} \frac{G_{t,\tau_{n+1}}^{N,\param_n}\addfunc{\tau_{n+1}}(\epart{t}{\ell})}{|\mathcal{L}_{t,\tau_{n+1}}^{\param_n}\1|_{\infty}}\right|^{p}\right]\\ = \mathbb{E}_{}\left[\mathbb{E}_{}\left[\left| \sum_{\ell=1}^{N} \ewght{t}{\ell} \frac{G_{t,\tau_{n+1}}^{N,\param_n}\addfunc{\tau_{n+1}}(\epart{t}{\ell})}{|\mathcal{L}_{t,\tau_{n+1}}^{\param_n}\1|_{\infty}}\right|^{p}\middle|\mathcal{F}_{n,t-1}^{N}\right]\frac{1}{b_{-}(\bfY_{t+T_n})^{p}}\right]
\end{multline*}
By Lemma~\ref{Lem:Upperbounds}\eqref{Lem:Upperbounds:ind}, using again the
Burkholder and convexity inequalities, there exists $C$
s.t.
\begin{align*}
\mathbb{E}_{}\left[\left| \sum_{\ell=1}^{N} \ewght{t}{\ell} \frac{G_{t,\tau_{n+1}}^{N,\param_n}\addfunc{\tau_{n+1}}(\epart{t}{\ell})}{|\mathcal{L}_{t,\tau_{n+1}}^{\param_n}\1|_{\infty}}\right|^{p}\middle|\mathcal{F}_{t-1,n}^{N}\right]&\leq CN^{(\frac{p}{2}-1)\vee 0}\mathbb{E}_{}\left[\sum_{\ell=1}^{N} \left|\ewght{t}{\ell} \frac{G_{t,\tau_{n+1}}^{N,\param_n}\addfunc{\tau_{n+1}}(\epart{t}{\ell})}{|\mathcal{L}_{t,\tau_{n+1}}^{\param_n}\1|_{\infty}}\right|^{p}\middle|\mathcal{F}_{n,t-1}^{N}\right]\\
&\leq CN^{\frac{p}{2}\vee 1}\mathbb{E}_{}\left[\left|\ewght{t}{1} \frac{G_{t,\tau_{n+1}}^{N,\param_n}\addfunc{\tau_{n+1}}(\epart{t}{1})}{|\mathcal{L}_{t,\tau_{n+1}}^{\param_n}\1|_{\infty}}\right|^{p}\middle|\mathcal{F}_{n,t-1}^{N}\right]\eqsp.
\end{align*}
The proof is concluded by \eqref{eq:bound-G}.
\end{proof}

\begin{proposition}
\label{Prop:NormC}
Assume A\ref{assum:strong} and A\ref{assum:dub:lec}. Let $\left\{
  (\epart{t}{\ell}, \ewght{t}{\ell}), 1 \leq \ell \leq N, 0 \leq t \leq
  \tau_{n+1} \right\}$ be the weighted samples obtained by
Algorithm~\ref{alg:FSMC} in Appendix~\ref{app:alg}, with input variables
$\param_n, \tau_{n+1}$, $N$, $\bfY_{T_n+1:T_n+\tau_{n+1}}$. For all $\bar p > 1$ and all $p\in (1,\bar p)$, there exists a constant $C$ s.t. for any $t\in\{0,\dots,\tau_{n+1}\}$,
\begin{multline}\label{Eq:NormC}
\mathbb{E}_{}\left[\left|C_{t,\tau_{n+1}}^{N,\param_n}(\addfunc{\tau_{n+1}})\right|^{p}\right]\leq CN^{(\frac{p}{2}\vee \frac{1}{\alpha})+(\frac{p}{2}\vee \frac{1}{\beta})-2p}\\
\times\mathbb{E}_{}\left[\left|\frac{\ewght{+}{}(\bfY_{t+T_n})}{b_{-}(\bfY_{t+T_{n}})}\right|^{2p}\mathbb{E}_{}\left[\left|\sum_{s=1}^{\tau_{n+1}}\rho^{|t-s|}\osc\{S(\cdot,\cdot, \bfY_{s+T_n})\}\right|^{\bar p}\middle|\mathcal{F}_{n,t-1}^{N}\right]^{p/\bar p}\right]\eqsp,
\end{multline}
where $C_{t,\tau}^{N,\param}$ is defined in \eqref{eq:defC} and $\alpha\eqdef \bar p/ p$ and $\beta^{-1} = 1 - \alpha^{-1}$.
\end{proposition}

\begin{proof}
Lemma~\ref{lem:esperancecond} applied with the function $h =
\mathcal{L}_{t,\tau_{n+1}}^{\param_n}\1$ and (\ref{eq:lcal}) yield for any $1
\leq \ell \leq N$
\[
  \CExpparam{\ewght{t}{\ell}
    \mathcal{L}_{t,\tau_{n+1}}^{\param_n}\1(\epart{t}{\ell})}{\mathcal{F}_{n,t-1}^{N}}{}
  =
  \frac{\phi_{t-1}^{N,\param_n}\left[\mathcal{L}_{t-1,\tau_{n+1}}^{\param_n}\1\right]}{\phi_{t-1}^{N,\param_n}[\adjfunc{t}{}{}]}\eqsp.
\]
Therefore, by definition of $C_{t,\tau}^{N,\param}$ (see~\eqref{eq:defC}),
$C_{t,\tau_{n+1}}^{N,\param_n}(\addfunc{\tau_{n+1}})$ is equal to
\begin{multline*}
  \frac{ \CExpparam{ A_{n,t}^N }{\mathcal{F}_{n,t-1}^{N}}{} - A_{n,t}^N}{
    \CExpparam{ A_{n,t}^N }{\mathcal{F}_{n,t-1}^{N}}{} \ A_{n,t}^N} \ 
  B_{n,t}^N \eqsp = \left( \CExpparam{ A_{n,t}^N
    }{\mathcal{F}_{n,t-1}^{N}}{} - A_{n,t}^N\right)  \cdots \\
  \times \frac{\Omega_{n,t}^N}{ \CExpparam{ A_{n,t}^N
    }{\mathcal{F}_{n,t-1}^{N}}{} \ A_{n,t}^N} \left( \frac{B_{n,t}^N}{
      \CExpparam{ \Omega_{n,t}^N }{\mathcal{F}_{n,t-1}^{N}}{} } +
    \frac{B_{n,t}^N}{\Omega_{n,t}^N\CExpparam{ \Omega_{n,t}^N
      }{\mathcal{F}_{n,t-1}^{N}}{}} \left( \CExpparam{ \Omega_{n,t}^N
      }{\mathcal{F}_{n,t-1}^{N}}{}- \Omega_{n,t}^N\right)\right) \\
  = B_{n,t}^N \left( \CExpparam{ A_{n,t}^N }{\mathcal{F}_{n,t-1}^{N}}{} -
    A_{n,t}^N\right) \frac{\Omega_{n,t}^N}{ \CExpparam{ A_{n,t}^N
    }{\mathcal{F}_{n,t-1}^{N}}{} \ A_{n,t}^N} \frac{1}{
    \CExpparam{ \Omega_{n,t}^N }{\mathcal{F}_{n,t-1}^{N}}{} }   \\
 + \left( \CExpparam{ A_{n,t}^N }{\mathcal{F}_{n,t-1}^{N}}{} -
    A_{n,t}^N \right) \left( \CExpparam{ \Omega_{n,t}^N
      }{\mathcal{F}_{n,t-1}^{N}}{}- \Omega_{n,t}^N\right)\frac{B_{n,t}^N}{ \CExpparam{ A_{n,t}^N
    }{\mathcal{F}_{n,t-1}^{N}}{} \ A_{n,t}^N} \frac{1}{
    \CExpparam{ \Omega_{n,t}^N }{\mathcal{F}_{n,t-1}^{N}}{} } 
\ \end{multline*}
with 
\begin{align*}
  A_{n,t}^N & \eqdef N^{-1}\sum_{\ell=1}^{N}\ewght{t}{\ell}
  \frac{\mathcal{L}_{t,\tau_{n+1}}^{\param_n}\1(\epart{t}{\ell})}{|\mathcal{L}_{t,\tau_{n+1}}^{\param_n}\1|_{\infty}}
  \eqsp, \\
  B_{n,t}^N & \eqdef \frac{1}{N} \sum_{\ell=1}^N \ewght{t}{\ell}
  \frac{G^{N,\param_n}_{t,\tau_{n+1}}\addfunc{\tau_{n+1}}(\epart{t}{\ell})}{|\mathcal{L}_{t,\tau_{n+1}}^{\param_n}\1|_\infty}
  \eqsp, \\
  \Omega_{n,t}^N & \eqdef \frac{1}{N} \sum_{\ell=1}^N \ewght{t}{\ell} \eqsp.
\end{align*}
This can be rewritten, 
\[
C_{t,\tau}^{N,\param} = C_{1} + C_{2}\eqsp,
\]
with
\[
C_{1} = B_{n,t}^N \left( \CExpparam{ A_{n,t}^N }{\mathcal{F}_{n,t-1}^{N}}{} -
    A_{n,t}^N\right) \frac{\Omega_{n,t}^N}{ \CExpparam{ A_{n,t}^N
    }{\mathcal{F}_{n,t-1}^{N}}{} \ A_{n,t}^N} \frac{1}{
    \CExpparam{ \Omega_{n,t}^N }{\mathcal{F}_{n,t-1}^{N}}{} }
\]
and
\begin{multline*}
C_{2} = \left( \CExpparam{ A_{n,t}^N }{\mathcal{F}_{n,t-1}^{N}}{} -
    A_{n,t}^N \right) \left( \CExpparam{ \Omega_{n,t}^N
      }{\mathcal{F}_{n,t-1}^{N}}{}- \Omega_{n,t}^N\right)\\
      \times\frac{B_{n,t}^N}{ \CExpparam{ A_{n,t}^N
    }{\mathcal{F}_{n,t-1}^{N}}{} \ A_{n,t}^N} \frac{1}{
    \CExpparam{ \Omega_{n,t}^N }{\mathcal{F}_{n,t-1}^{N}}{} } \eqsp.
\end{multline*}
By Lemmas~\ref{lem:esperancecond} and Lemmas~\ref{Lem:Upperbounds}\eqref{Lem:Upperbounds:boundlow}, and A\ref{assum:dub:lec},
\[
\frac{1}{ \CExpparam{ \Omega_{n,t}^N }{\mathcal{F}_{n,t-1}^{N}}{} } \leq \frac{ \sigma_-|\adjfunc{}{}{}|_\infty}{b_{-}(\bfY_{t+T_n})} \eqsp; \qquad 
\frac{\Omega_{t,n}^N}{\CExpparam{ A_{n,t}^N }{\mathcal{F}_{n,t-1}^{N}}{} \ 
  A_{n,t}^N } \leq \left(\frac{\sigma_+}{\sigma_-}\right)^2 \frac{|\adjfunc{}{}{}|_\infty}{\sigma_-
b_{-}(\bfY_{t+T_n})};
\]
and by Lemma~\ref{Lem:Upperbounds}\eqref{Lem:Upperbounds:boundup}
\[
\frac{B_{n,t}^N}{\CExpparam{ A_{n,t}^N }{\mathcal{F}_{n,t-1}^{N}}{} \ 
  A_{n,t}^N } \leq \left(\frac{\sigma_+}{\sigma_-}\right)^2 \frac{|\adjfunc{}{}{}|_\infty}{\sigma_-
b_{-}(\bfY_{t+T_n})} \left( \sum_{s=1}^{\tau_{n+1}} \rho^{|t-s|} \osc(S(\cdot,\cdot,\bfY_{s+T_n}) \right) \eqsp.
\]
Therefore, there exists a constant $C$ s.t.
\begin{align*}
\mathbb{E}_{}\left[\left|C_{1}\right|^{p}\middle|\mathcal{F}_{n,t-1}^{N}\right]\leq C\left|\frac{1}{b_{-}(\bfY_{t+T_{n}})}\right|^{2p}\mathbb{E}_{}\left[\left|B_{n,t}^{N}\right|^{p}\left|\CExpparam{ A_{n,t}^N }{\mathcal{F}_{n,t-1}^{N}}{} \ 
 - A_{n,t}^N\right|^{p}\middle|\mathcal{F}_{n,t-1}^{N}\right]
\end{align*}
Applying the Holder inequality with $\alpha \eqdef \bar p/p\geq 1$ and $\beta^{-1} \eqdef 1-\alpha^{-1}$ yields
\begin{multline*}
\mathbb{E}_{}\left[\left|B_{n,t}^{N}\right|^{p}\left|\CExpparam{ A_{n,t}^N }{\mathcal{F}_{n,t-1}^{N}}{} \ 
 - A_{n,t}^N\right|^{p}\middle|\mathcal{F}_{n,t-1}^{N}\right]\\
 \leq \mathbb{E}_{}\left[\left|B_{n,t}^{N}\right|^{\alpha p}\middle|\mathcal{F}_{n,t-1}^{N}\right]^{1/\alpha}\mathbb{E}_{}\left[\left|\CExpparam{ A_{n,t}^N }{\mathcal{F}_{n,t-1}^{N}}{} \ 
 - A_{n,t}^N\right|^{\beta p}\middle|\mathcal{F}_{n,t-1}^{N}\right]^{1/\beta}\eqsp.
\end{multline*}
By Proposition~\ref{Prop:NormD}, 
\begin{align*}
\mathbb{E}_{}\left[\left|B_{n,t}^{N}\right|^{\alpha p}\middle|\mathcal{F}_{n,t-1}^{N}\right]^{1/\alpha}\\
&\hspace{-2cm}\leq CN^{(\frac{p}{2}\vee \frac{1}{\alpha})-p}\CExp{\left|\ewght{t}{1} \frac{G_{t,\tau_{n+1}}^{N,\param_n}\addfunc{\tau_{n+1}}(\epart{t}{1})}{|\mathcal{L}_{t,\tau_{n+1}}^{\param_n}\1|_{\infty}}\right|^{\alpha p}}{\mathcal{F}_{n,t-1}^{N}}^{1/\alpha}\\
&\hspace{-2cm}\leq  CN^{(\frac{p}{2}\vee \frac{1}{\alpha})-p}\ewght{+}{}(\bfY_{t+T_n})^{p}\mathbb{E}_{}\left[\left|\sum_{s=1}^{\tau_{n+1}}\rho^{|t-s|}\osc\{S(\cdot,\cdot, \bfY_{s+T_n})\}\right|^{\alpha p}\middle|\mathcal{F}_{n,t-1}^{N}\right]^{1/\alpha}\eqsp,
\end{align*}

Given $\mathcal{F}_{n,t-1}^{N}$, the random variables
$\left\{\mathbb{E}_{}\left[\left.\ewght{t}{\ell}
      \frac{\mathcal{L}_{t,\tau_{n+1}}^{\param_n}\1(\epart{t}{1})}{|\mathcal{L}_{t,\tau_{n+1}}^{\param_n}\1|_\infty}\right|\mathcal{F}_{n,t-1}^{N}\right]-
  \ewght{t}{\ell}
  \frac{\mathcal{L}_{t,\tau_{n+1}}\1(\epart{t}{\ell})}{|\mathcal{L}_{t,\tau_{n+1}}^{\param_n}\1|_\infty}\right\}_{\ell=1}^{N}$
are conditionally independent, centered and bounded by Lemma~\ref{Lem:Upperbounds}.
Following the same steps as in the proof of Proposition \ref{Prop:NormD}, there exists a constant $C$ such that
\begin{equation*}
\mathbb{E}_{}\left[\left|\CExpparam{ A_{n,t}^N }{\mathcal{F}_{n,t-1}^{N}}{} \ 
  -A_{n,t}^N\right|^{\beta p}\middle|\mathcal{F}_{n,t-1}^{N}\right]^{1/\beta}\leq CN^{(\frac{p}{2}\vee \frac{1}{\beta})-p}\ewght{+}{}(\bfY_{t+T_n})^{p}\eqsp.
\end{equation*}
Hence,
\begin{multline*}
\mathbb{E}_{}\left[\left|C_{1}\right|^{p}\right]\leq CN^{(\frac{p}{2}\vee \frac{1}{\alpha})+(\frac{p}{2}\vee \frac{1}{\beta})-2p}\\
\times\mathbb{E}_{}\left[\left|\frac{\ewght{+}{}(\bfY_{t+T_n})}{b_{-}(\bfY_{t+T_{n}})}\right|^{2p}\mathbb{E}_{}\left[\left|\sum_{s=1}^{\tau_{n+1}}\rho^{|t-s|}\osc\{S(\cdot,\cdot, \bfY_{s+T_n})\}\right|^{\bar p}\middle|\mathcal{F}_{n,t-1}^{N}\right]^{p/\bar p}\right]\eqsp.
\end{multline*}
Similarly, using
\[
\mathbb{E}_{}\left[\left| \CExpparam{ \Omega_{t,n}^N
    }{\mathcal{F}_{n,t-1}^{N}}{} -  \Omega_{t,n}^N\right|^{\alpha p}\middle|\mathcal{F}_{n,t-1}^{N}\right]^{1/\alpha}\leq CN^{(\frac{p}{2}\vee \frac{1}{\alpha})-p}\ewght{+}{}(\bfY_{t+T_n})^{p}\eqsp,
\]
yields
\begin{multline*}
\mathbb{E}_{}\left[\left|C_{2}\right|^{p}\right]\leq CN^{-p}\\
\times\mathbb{E}_{}\left[\left|\frac{\ewght{+}{}(\bfY_{t+T_n})}{b_{-}(\bfY_{t+T_{n}})}\right|^{2p}\mathbb{E}_{}\left[\left|\sum_{s=1}^{\tau_{n+1}}\rho^{|t-s|}\osc\{S(\cdot,\cdot, \bfY_{s+T_n})\}\right|^{\bar p}\middle|\mathcal{F}_{n,t-1}^{N}\right]^{p/\bar p}\right]\eqsp.
\end{multline*}
 \end{proof}
\bibliographystyle{plain}
\bibliography{./onlineblock}

\end{document}